\documentclass[a4paper, reqno]{amsart}
\usepackage{graphicx}
\usepackage[bookmarksnumbered,colorlinks,plainpages]{hyperref}

\usepackage{float}
\usepackage{amsthm}
\usepackage[all,2cell]{xy}
\usepackage{amsfonts}
\usepackage{amsmath}
\usepackage{amssymb}
\usepackage{xypic,leqno,amscd,amssymb,pstricks,latexsym,amsbsy,xypic,mathrsfs,verbatim}
\usepackage{multirow}
\usepackage{booktabs} 
\usepackage{makecell} 

\usepackage{etex}
\usepackage{mathtools}

\usepackage{tikz}
\usetikzlibrary{positioning,arrows}
\usetikzlibrary{decorations.pathreplacing}

\UseAllTwocells \SilentMatrices
\newtheorem{thm}{Theorem}[section]

\newtheorem{cor}[thm]{Corollary}
\newtheorem{lem}[thm]{Lemma}

\newtheorem{prop}[thm]{Proposition}
\theoremstyle{definition}

\theoremstyle{remark}

\numberwithin{equation}{section}

\def\End{{\rm End}}

\def\bbX{{\mathbb X}}

\renewcommand{\mod}{\operatorname{mod}\nolimits}

\begin{document}
\title[Derived equivalences between one-branch extensions of ``rectangles'']{Derived equivalences between one-branch extension algebras of ``rectangles''}

\author[Qiang Dong, Yanan Lin, Shiquan Ruan] {Qiang Dong, Yanan Lin, Shiquan Ruan$^*$}

\address{Qiang Dong, Yanan Lin and Shiquan Ruan;\; School of Mathematical Sciences, Xiamen University, Xiamen, 361005, Fujian, P.R. China.}
\email{dongqiang@stu.xmu.edu.cn, ynlin@xmu.edu.cn, sqruan@xmu.edu.cn}

\thanks{$^*$ the corresponding author}
\date{\today}
\maketitle

\begin{abstract}
In this paper we investigate the incidence algebras arising from one-branch extensions of ``rectangles''. There are four different ways to form such extensions, and all four kinds of incidence algebras turn out to be derived equivalent. We provide realizations for all of them by tilting complexes in a Nakayama algebra. As an application, we obtain the explicit formulas of the Coxeter polynomials for a half of Nakayama algebras (i.e., the Nakayama algebras $N(n,r)$ with $2r\geq n+2$). Meanwhile, an unexpected derived equivalence between Nakayama algebras $N(2r-1,r)$ and $N(2r-1,r+1)$ has been found.
\end{abstract}

\section{Introduction}

Let ${\bf k}$ denote an algebraically closed field.
For the terminology ``rectangle'' we refer to the following bound quiver (with full commutative relations) for $u\geq 1$:
$$
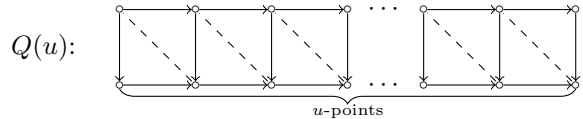
\begin{tikzpicture}
\node()at(-1,0.5){$Q(u)$:};
\node()at(0,0){\tiny{$\circ$}};
\node()at(1,0){\tiny{$\circ$}};
\node()at(1,1){\tiny{$\circ$}};
\node()at(0,1){\tiny{$\circ$}};
\node()at(2,0){\tiny{$\circ$}};
\node()at(2,1){\tiny{$\circ$}};
\node()at(3,0){\tiny{$\circ$}};
\node()at(3,1){\tiny{$\circ$}};
\node()at(4,0){\tiny{$\circ$}};
\node()at(4,1){\tiny{$\circ$}};
\node()at(5,0){\tiny{$\circ$}};
\node()at(5,1){\tiny{$\circ$}};
\node()at(6,0){\tiny{$\circ$}};
\node()at(6,1){\tiny{$\circ$}};
\node()at(3.5,0){$\cdots$};
\node()at(3.5,1){$\cdots$};
\node()at(9,0){};
\draw[-,dashed](0.95,0.05)--(0.05,0.95);
\draw[-,dashed](1.95,0.05)--(1.05,0.95);
\draw[-,dashed](2.95,0.05)--(2.05,0.95);
\draw[-,dashed](4.95,0.05)--(4.05,0.95);
\draw[-,dashed](5.95,0.05)--(5.05,0.95);
\draw[->](0,0.95)--(0,0.05);
\draw[->](1,0.95)--(1,0.05);
\draw[->](0.05,0)--(0.95,0);
\draw[->](0.05,1)--(0.95,1);
\draw[->](2,0.95)--(2,0.05);
\draw[->](3,0.95)--(3,0.05);
\draw[->](4,0.95)--(4,0.05);
\draw[->](1.05,0)--(1.95,0);
\draw[->](1.05,1)--(1.95,1);
\draw[->](2.05,0)--(2.95,0);
\draw[->](2.05,1)--(2.95,1);
\draw[->](4.05,0)--(4.95,0);
\draw[->](4.05,1)--(4.95,1);
\draw[->](5,0.95)--(5,0.05);
\draw[->](5.05,0)--(5.95,0);
\draw[->](5.05,1)--(5.95,1);
\draw[->](6,0.95)--(6,0.05);
\draw[decorate,decoration={brace,mirror,amplitude=2mm}](0,-0.05)--(6,-0.05);
\node()at(3,-0.35){\tiny{$u$-points}};
\end{tikzpicture}$$
The bound quiver $Q(u)$ and its path algebra $A(u)$ over ${\bf k}$, have various appearances in different branches of representation theory, such as singularity theory, incidence algebras of posets, monomorphism category and stable category of vector bundles, etc. c.f. \cite{HM2014, KLM2013, KLM20132, Lad2013, RS2008}.

The path algebra $A(u)$ can be thought as the tensor product $\vec{A}_2\otimes \vec{A}_{u}$ of path algebras of Dynkin type $\mathbb{A}$. In the remarkable paper \cite{Lad2013}, Ladkani revealed an interesting connection between the ``rectangles'' with ``lines'' and ``triangles''. More precisely, $A(u)$ (rectangle) is derived equivalent to the Nakayama algebra $N(2u, 3)$ (line), and also to the stable Auslander algebra of Dynkin type $\vec{A}_{u}$ (triangle).

The relevance of $Q(u)$ and $A(u)$ to weighted projective lines is demonstrated in \cite{KLM2013, KLM20132}. For a weighted projective line $\bbX$ of weight type $(2,3,u+1)$, Kussin, Lenzing and Meltzer constructed a tilting bundle in the stable category of vector bundle over $\bbX$, whose endomorphism algebra is isomorphic to $A(u)$ (\cite{KLM2013}). Moreover, in \cite{KLM20132} they established a relation between the category of vector bundles over $\bbX$ with the category of submodules of nilpotent linear operators studied by Ringel and Schmidmeier \cite{RS2008}, yielding an equivalence between the bounded derived category of $A(u)$ and the stable category of submodules of nilpotent linear operators of nilpotent degree $u+1$.

In this paper, we investigate the one-branch extension algebras of $A(u)$, arising from $A(u)$ by taking one-point extension (resp. co-extension) iteratively.

Let $A(u)_v,\ A(u)^v,\ _vA(u)$ and $^vA(u)$ be the one-branch extension algebras of $A(u)$ associated to the bound quivers in Figure \ref{algebraquiver}, respectively. It turns out that they are all derived equivalent.

\begin{figure}[H]
\begin{tikzpicture}
\node()at(-1,0.5){$Q(u)_v$:};
\node()at(0,0){\tiny{$\circ$}};
\node()at(1,0){\tiny{$\circ$}};
\node()at(1,1){\tiny{$\circ$}};
\node()at(0,1){\tiny{$\circ$}};
\node()at(3,1){\tiny{$\circ$}};
\node()at(3,0){\tiny{$\circ$}};
\node()at(2,0){\tiny{$\circ$}};
\node()at(2,1){\tiny{$\circ$}};
\node()at(4,0){\tiny{$\circ$}};
\node()at(4,1){\tiny{$\circ$}};
\node()at(5,0){\tiny{$\circ$}};
\node()at(5,1){\tiny{$\circ$}};
\node()at(2.5,0){$\cdots$};
\node()at(2.5,1){$\cdots$};
\draw[-,dashed](0.95,0.05)--(0.05,0.95);
\draw[-,dashed](1.95,0.05)--(1.05,0.95);
\draw[-,dashed](3.95,0.05)--(3.05,0.95);
\draw[-,dashed](4.95,0.05)--(4.05,0.95);
\draw[->](0,0.95)--(0,0.05);
\draw[->](1,0.95)--(1,0.05);
\draw[->](2,0.95)--(2,0.05);
\draw[->](3,0.95)--(3,0.05);
\draw[->](0.05,0)--(0.95,0);
\draw[->](0.05,1)--(0.95,1);
\draw[->](1.05,0)--(1.95,0);
\draw[->](1.05,1)--(1.95,1);
\draw[->](3.05,0)--(3.95,0);
\draw[->](3.05,1)--(3.95,1);
\draw[->](4,0.95)--(4,0.05);
\draw[->](4.05,0)--(4.95,0);
\draw[->](4.05,1)--(4.95,1);
\draw[->](5,0.95)--(5,0.05);
\draw[decorate,decoration={brace,mirror,amplitude=2mm}](0,-0.05)--(5,-0.05);
\node()at(2.5,-0.35){\tiny{$u$-points}};
\node()at(6,0){\tiny{$\circ$}};
\node()at(7.5,0){$\cdots$};
\node()at(9,0){\tiny{$\circ$}};
\draw[<-](5.95,0)--(5.05,0);
\draw[<-](6.95,0)--(6.05,0);
\draw[<-](8.95,0)--(8.05,0);
\node()at(7,0){\tiny{$\circ$}};
\node()at(8,0){\tiny{$\circ$}};
\draw[decorate,decoration={brace,mirror,amplitude=2mm}](6,-0.05)--(9,-0.05);
\node()at(7.5,-0.35){\tiny{$v$-points}};
\end{tikzpicture}

\begin{tikzpicture}
\node()at(-1,0.5){$Q(u)^v$:};
\node()at(0,0){\tiny{$\circ$}};
\node()at(1,0){\tiny{$\circ$}};
\node()at(1,1){\tiny{$\circ$}};
\node()at(0,1){\tiny{$\circ$}};
\node()at(3,1){\tiny{$\circ$}};
\node()at(3,0){\tiny{$\circ$}};
\node()at(2,0){\tiny{$\circ$}};
\node()at(2,1){\tiny{$\circ$}};
\node()at(4,0){\tiny{$\circ$}};
\node()at(4,1){\tiny{$\circ$}};
\node()at(5,0){\tiny{$\circ$}};
\node()at(5,1){\tiny{$\circ$}};
\node()at(2.5,0){$\cdots$};
\node()at(2.5,1){$\cdots$};
\draw[-,dashed](0.95,0.05)--(0.05,0.95);
\draw[-,dashed](1.95,0.05)--(1.05,0.95);
\draw[-,dashed](3.95,0.05)--(3.05,0.95);
\draw[-,dashed](4.95,0.05)--(4.05,0.95);
\draw[->](0,0.95)--(0,0.05);
\draw[->](1,0.95)--(1,0.05);
\draw[->](2,0.95)--(2,0.05);
\draw[->](3,0.95)--(3,0.05);
\draw[->](0.05,0)--(0.95,0);
\draw[->](0.05,1)--(0.95,1);
\draw[->](1.05,0)--(1.95,0);
\draw[->](1.05,1)--(1.95,1);
\draw[->](3.05,0)--(3.95,0);
\draw[->](3.05,1)--(3.95,1);
\draw[->](4,0.95)--(4,0.05);
\draw[->](4.05,0)--(4.95,0);
\draw[->](4.05,1)--(4.95,1);
\draw[->](5,0.95)--(5,0.05);
\draw[decorate,decoration={brace,mirror,amplitude=2mm}](0,-0.05)--(5,-0.05);
\node()at(2.5,-0.35){\tiny{$u$-points}};
\node()at(6,1){\tiny{$\circ$}};
\node()at(7.5,1){$\cdots$};
\node()at(9,1){\tiny{$\circ$}};
\node()at(7,1){\tiny{$\circ$}};
\node()at(8,1){\tiny{$\circ$}};
\draw[<-](5.95,1)--(5.05,1);
\draw[<-](6.95,1)--(6.05,1);
\draw[<-](8.95,1)--(8.05,1);
\draw[decorate,decoration={brace,mirror,amplitude=2mm}](6,0.95)--(9,0.95);
\node()at(7.5,0.65){\tiny{$v$-points}};
\end{tikzpicture}

\begin{tikzpicture}
\node()at(-5,0.5){$_vQ(u)$:};
\node()at(0,0){\tiny{$\circ$}};
\node()at(1,0){\tiny{$\circ$}};
\node()at(1,1){\tiny{$\circ$}};
\node()at(0,1){\tiny{$\circ$}};
\node()at(3,1){\tiny{$\circ$}};
\node()at(3,0){\tiny{$\circ$}};
\node()at(2,0){\tiny{$\circ$}};
\node()at(2,1){\tiny{$\circ$}};
\node()at(4,0){\tiny{$\circ$}};
\node()at(4,1){\tiny{$\circ$}};
\node()at(5,0){\tiny{$\circ$}};
\node()at(5,1){\tiny{$\circ$}};
\node()at(2.5,0){$\cdots$};
\node()at(2.5,1){$\cdots$};
\draw[-,dashed](0.95,0.05)--(0.05,0.95);
\draw[-,dashed](1.95,0.05)--(1.05,0.95);
\draw[-,dashed](3.95,0.05)--(3.05,0.95);
\draw[-,dashed](4.95,0.05)--(4.05,0.95);
\draw[->](0,0.95)--(0,0.05);
\draw[->](1,0.95)--(1,0.05);
\draw[->](2,0.95)--(2,0.05);
\draw[->](3,0.95)--(3,0.05);
\draw[->](0.05,0)--(0.95,0);
\draw[->](0.05,1)--(0.95,1);
\draw[->](1.05,0)--(1.95,0);
\draw[->](1.05,1)--(1.95,1);
\draw[->](3.05,0)--(3.95,0);
\draw[->](3.05,1)--(3.95,1);
\draw[->](4,0.95)--(4,0.05);
\draw[->](4.05,0)--(4.95,0);
\draw[->](4.05,1)--(4.95,1);
\draw[->](5,0.95)--(5,0.05);
\draw[decorate,decoration={brace,mirror,amplitude=2mm}](0,-0.05)--(5,-0.05);
\node()at(2.5,-0.35){\tiny{$u$-points}};
\node()at(-1,0){\tiny{$\circ$}};
\node()at(-2,0){\tiny{$\circ$}};
\node()at(-3,0){\tiny{$\circ$}};
\node()at(-2.5,0){$\cdots$};
\node()at(-4,0){\tiny{$\circ$}};
\draw[->](-0.95,0)--(-0.05,0);
\draw[->](-1.95,0)--(-1.05,0);
\draw[->](-3.95,0)--(-3.05,0);
\draw[decorate,decoration={brace,mirror,amplitude=2mm}](-4,-0.05)--(-1,-0.05);
\node()at(-2.5,-0.35){\tiny{$v$-points}};
\end{tikzpicture}

\begin{tikzpicture}
\node()at(-5,0.5){$^vQ(u)$:};
\node()at(-2,1){\tiny{$\circ$}};
\node()at(-3,1){\tiny{$\circ$}};
\node()at(0,0){\tiny{$\circ$}};
\node()at(1,0){\tiny{$\circ$}};
\node()at(1,1){\tiny{$\circ$}};
\node()at(0,1){\tiny{$\circ$}};
\node()at(3,1){\tiny{$\circ$}};
\node()at(3,0){\tiny{$\circ$}};
\node()at(2,0){\tiny{$\circ$}};
\node()at(2,1){\tiny{$\circ$}};
\node()at(4,0){\tiny{$\circ$}};
\node()at(4,1){\tiny{$\circ$}};
\node()at(5,0){\tiny{$\circ$}};
\node()at(5,1){\tiny{$\circ$}};
\node()at(2.5,0){$\cdots$};
\node()at(2.5,1){$\cdots$};
\draw[-,dashed](0.95,0.05)--(0.05,0.95);
\draw[-,dashed](1.95,0.05)--(1.05,0.95);
\draw[-,dashed](3.95,0.05)--(3.05,0.95);
\draw[-,dashed](4.95,0.05)--(4.05,0.95);
\draw[->](0,0.95)--(0,0.05);
\draw[->](1,0.95)--(1,0.05);
\draw[->](2,0.95)--(2,0.05);
\draw[->](3,0.95)--(3,0.05);
\draw[->](0.05,0)--(0.95,0);
\draw[->](0.05,1)--(0.95,1);
\draw[->](1.05,0)--(1.95,0);
\draw[->](1.05,1)--(1.95,1);
\draw[->](3.05,0)--(3.95,0);
\draw[->](3.05,1)--(3.95,1);
\draw[->](4,0.95)--(4,0.05);
\draw[->](4.05,0)--(4.95,0);
\draw[->](4.05,1)--(4.95,1);
\draw[->](5,0.95)--(5,0.05);
\draw[decorate,decoration={brace,mirror,amplitude=2mm}](0,-0.05)--(5,-0.05);
\node()at(2.5,-0.35){\tiny{$u$-points}};
\node()at(-1,1){\tiny{$\circ$}};
\node()at(-2.5,1){$\cdots$};
\node()at(-4,1){\tiny{$\circ$}};
\draw[->](-0.95,1)--(-0.05,1);
\draw[->](-1.95,1)--(-1.05,1);
\draw[->](-3.95,1)--(-3.05,1);
\draw[decorate,decoration={brace,mirror,amplitude=2mm}](-4,0.95)--(-1,0.95);
\node()at(-2.5,0.65){\tiny{$v$-points}};
\end{tikzpicture}
\caption{One-branch extensions of ``rectangles''}\label{algebraquiver}
\end{figure}
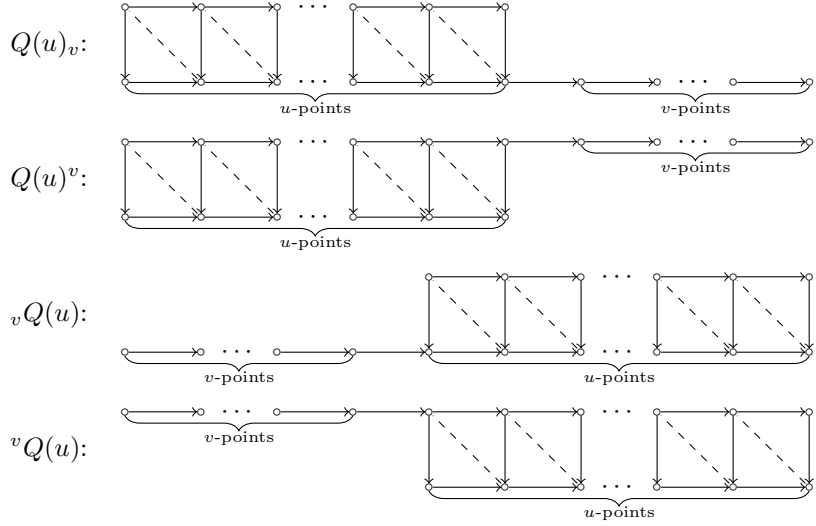

\begin{thm}\label{theorem}{\textup{(see Proposition \ref{main1} and Proposition \ref{main2})}}
For any $u\geq 1$ and $v\geq 1$, we have the following derived equivalences:
$$\textup{D}^{b}(A(u)_v)\simeq\textup{D}^{b}(A(u)^v)\simeq\textup{D}^{b}(_vA(u))\simeq\textup{D}^{b}(^vA(u)).
$$
\end{thm}

All these four kinds of algebras are derived equivalent to the Nakayama algebra $N(2u+v,u+v+1)$, but the explicit derived functors are not clear. This problem has been solved in Section 5. We provide a realization for each of them by a tilting complex over the Nakayama algebra $N(2u+v,u+v+1)$.
As an application, we obtain the explicit formulas of the Coxeter polynomials for a half of Nakayama algebras, which is interesting in its own right.

\begin{prop}\label{coxeterpolyofNakayama}{\textup{(See Corollary \ref{corollaryofpolynomial})}}
For $r,n\in\mathbb{N}$ with $2r\geq n+2$, the Coxeter polynomial $\chi_{N(n,r)}(\lambda)$ of the Nakayama algebra $N(n,r)$ is given as below:
\begin{itemize}
\item[(1)] For $2r=n+2$,
$$\chi_{N(n,r)}(\lambda)=\left\{\begin{array}{lcl}\frac{(\lambda+1)\big(\lambda^{3n-3r+6}-(-1)^{n-r}\big)}
{(\lambda^3+1)\big(\lambda^{n-r+2}-(-1)^{n-r}\big)},&&n-r\not\equiv 1\,(\mod 3),\\\\
\frac{(\lambda+1)\big(\lambda^{n-r+2}-(-1)^{n-r}\big)^2}{\lambda^3+1},&&n-r\equiv 1\,(\mod 3).\end{array}\right.$$
\item[(2)] For $2r\geq n+3$,
$$\chi_{N(n,r)}(\lambda)=\left\{\begin{array}{lcl}\frac{(\lambda+1)
\big(\lambda^{n+2}+(-1)^{n-r}\sum\limits_{j=0}^{2r-n-2}\lambda^{n-r+2+j}+1\big)}{\lambda^3+1},
&&n-r\equiv 0\,(\mod 3),\\\\
\frac{(\lambda+1)\big(\lambda^{n-r+2}-(-1)^{n-r}\big)\big(\lambda^{r}-(-1)^{n-r}\big)}{\lambda^3+1},
&&n-r\equiv 1\,(\mod 3),\\\\
\frac{(\lambda+1)\big(\lambda^{n+2}+(-1)^{n-r+1}\sum\limits_{j=2}^{2r-n-2}\lambda^{n-r+1+j}+1\big)}{\lambda^3+1},
&&n-r\equiv 2\,(\mod 3).\end{array}\right.$$
\end{itemize}
\end{prop}

For the proof of Theorem \ref{theorem}, the powerful derived equivalence between categories of sheaves over finite posets developed by Ladkani \cite{Lad2008} has been used, and it also relies on the explicit realization of $A(u)$ by tilting bundles in the stable category of vector bundles over the weighted projective line $\bbX$ of weight type $(2,3,u+1)$. And the proof of Proposition \ref{coxeterpolyofNakayama} is basing on the one-point extension approach and perpendicular calculations, c.f., Happel \cite{Hap2009} and Lenzing \cite{Len1997}.

The paper is organized as follows. In Section 2, we recall basic definitions and facts on incidence algebras and weighted projective lines. In Section 3, we give the Coxeter polynomials for the one-branch extension algebra $A(u)_v$ by using one-point extension approach and perpendicular calculations. The derived equivalences between four kinds of one-branch extension algebras are proved in Section 4. As an application, we obtain the explicit formulas of the Coxeter polynomials for all the Nakayama algebra $N(n,r)$ with $2r\geq n+2$, and an unexpected derived equivalence between the Nakayama algebras $N(2r-1,r)$ and $N(2r-1,r+1)$ is given. In the final Section 5, we give explicit realizations of $A(u)_v,\ A(u)^v,\ _vA(u)$ and $^vA(u)$ by tilting complexes over the Nakayama algebra $N(2u+v,u+v+1)$.

\section{preliminary}

We recall some notations and facts on incidence algebras of posets and weighted projective lines in this subsection.

\subsection{Incidence algebras of posets}

A set $X$ with a binary relation $\leq$ is a \emph{partially ordered set} if $\leq$ is reflexive, transitive and antisymmetric. Throughout this paper, the term \emph{poset} will mean a finite partially ordered set.

A poset can be represented visually by \emph{Hasse diagrams}. Such a diagram consists of vertices which represent the elements of the poset, and line segments which determine the relationships between elements. A line segment is drawn from $x$ down to $y$ if $x\leq y$, $x\neq y$ and whenever $x\leq z\leq y$ with $y\neq z$, then $x=z$. In this paper, we focus on posets with the following Hasse diagrams.

\begin{figure}[H]\begin{tikzpicture}
\node()at(0,0){\tiny{$\bullet$}};
\node()at(-0.25,0.5){\tiny{$\bullet$}};
\node()at(-1,2){\tiny{$\bullet$}};
\node()at(-1.25,2.5){\tiny{$\bullet$}};
\node()at(-0.75,1.5){\tiny{$\bullet$}};
\node()at(-0.5,1){\tiny{$\bullet$}};
\node()at(0,1.25){\tiny{$\bullet$}};
\node()at(-0.25,1.75){\tiny{$\bullet$}};
\node()at(1.25,1.5){\tiny{$\bullet$}};
\node()at(1.5,1){\tiny{$\bullet$}};
\node()at(2,1.25){\tiny{$\bullet$}};
\node()at(1.75,1.75){\tiny{$\bullet$}};
\draw[-](0,0)--(-0.25,0.5);
\draw[dotted](-0.5,1)--(-0.75,1.5);
\draw[-](-0.25,0.5)--(-0.5,1);
\draw[-](-1,2)--(-0.75,1.5);
\draw[-](-1,2)--(-1.25,2.5);
\node()at(0.5,0.25){\tiny{$\bullet$}};
\node()at(0.25,0.75){\tiny{$\bullet$}};
\node()at(-0.5,2.25){\tiny{$\bullet$}};
\node()at(-0.75,2.75){\tiny{$\bullet$}};
\draw[-](-0.5,2.25)--(-1,2);
\draw[-](-0.75,2.75)--(-1.25,2.5);
\draw[-](0.5,0.25)--(0.25,0.75);
\node()at(0.5,-1){\tiny{$\bullet$}};
\node()at(0.75,-1.5){\tiny{$\bullet$}};
\node()at(3,-0.75){\tiny{$\bullet$}};
\node()at(3.25,-1.25){\tiny{$\bullet$}};
\draw[dotted](0,1.25)--(-0.25,1.75);
\draw[-](0.25,0.75)--(0,1.25);
\draw[-](0,0)--(0.5,0.25);
\draw[-](-0.25,0.5)--(0.25,0.75);
\draw[-](-0.5,1)--(0,1.25);
\draw[-](-0.75,1.5)--(-0.25,1.75);
\draw[-](-0.5,2.25)--(-0.25,1.75);
\draw[-](-0.5,2.25)--(-0.75,2.75);
\node()at(0.25,-0.5){\tiny{$\bullet$}};
\node()at(1,-2){\tiny{$\bullet$}};
\draw[-](0,0)--(0.25,-0.5);
\draw[-](0.5,-1)--(0.25,-0.5);
\draw[-](1,-2)--(0.75,-1.5);
\draw[dotted](0.75,-1.5)--(0.5,-1);
\node()at(2,0){\tiny{$\bullet$}};
\node()at(1.75,0.5){\tiny{$\bullet$}};
\node()at(1,2){\tiny{$\bullet$}};
\node()at(0.75,2.5){\tiny{$\bullet$}};
\draw[-](0.75,2.5)--(1.25,2.75);
\draw[-](1,2)--(1.5,2.25);
\draw[-](0.75,2.5)--(1,2);
\draw[-](1,2)--(1.25,1.5);
\draw[-](2,0)--(1.75,0.5);
\draw[dotted](1.5,1)--(1.25,1.5);
\draw[-](1.75,0.5)--(1.5,1);
\node()at(2.5,0.25){\tiny{$\bullet$}};
\node()at(2.25,0.75){\tiny{$\bullet$}};
\node()at(1.5,2.25){\tiny{$\bullet$}};
\node()at(1.25,2.75){\tiny{$\bullet$}};
\draw[-](1.25,2.75)--(1.5,2.25);
\draw[-](1.75,1.75)--(1.5,2.25);
\draw[-](2.5,0.25)--(2.25,0.75);
\draw[dotted](2,1.25)--(1.75,1.75);
\draw[-](2.25,0.75)--(2,1.25);
\draw[-](2,0)--(2.5,0.25);
\draw[-](1.75,0.5)--(2.25,0.75);
\draw[-](1.5,1)--(2,1.25);
\draw[-](1.25,1.5)--(1.75,1.75);
\node()at(2.75,-0.25){\tiny{$\bullet$}};
\node()at(3.5,-1.75){\tiny{$\bullet$}};
\draw[-](2.5,0.25)--(2.75,-0.25);
\draw[-](3,-0.75)--(2.75,-0.25);
\draw[-](3.25,-1.25)--(3.5,-1.75);
\draw[dotted](3,-0.75)--(3.25,-1.25);
\node()at(5.5,-1.75){\tiny{$\bullet$}};
\node()at(5.25,-1.25){\tiny{$\bullet$}};
\node()at(4.5,0.25){\tiny{$\bullet$}};
\node()at(4.25,0.75){\tiny{$\bullet$}};
\node()at(4,1.25){\tiny{$\bullet$}};
\node()at(3.25,2.75){\tiny{$\bullet$}};
\draw[-](4,1.25)--(3.75,1.75);
\draw[-](4,1.25)--(4.25,0.75);
\draw[-](3.25,2.75)--(3.5,2.25);
\draw[dotted](3.75,1.75)--(3.5,2.25);
\draw[-](5.5,-1.75)--(5.25,-1.25);
\draw[-](5.25,-1.25)--(5,-0.75);
\draw[-](4.5,0.25)--(4.75,-0.25);
\draw[-](4.5,0.25)--(4.25,0.75);
\draw[dotted](5,-0.75)--(4.75,-0.25);
\draw[-](5.5,-1.75)--(6,-1.5);
\draw[-](5.25,-1.25)--(5.75,-1);
\draw[-](4.5,0.25)--(5,0.5);
\draw[-](4.75,1)--(4.25,0.75);
\draw[-](5,-0.75)--(5.5,-0.5);
\draw[-](4.75,-0.25)--(5.25,0);
\node()at(6,-1.5){\tiny{$\bullet$}};
\node()at(5.75,-1){\tiny{$\bullet$}};
\node()at(5,0.5){\tiny{$\bullet$}};
\node()at(4.75,1){\tiny{$\bullet$}};
\draw[-](6,-1.5)--(5.75,-1);
\draw[-](5.75,-1)--(5.5,-0.5);
\draw[-](5,0.5)--(5.25,0);
\draw[-](5,0.5)--(4.75,1);
\node()at(5.25,0){\tiny{$\bullet$}};
\node()at(5.5,-0.5){\tiny{$\bullet$}};
\node()at(4.75,-0.25){\tiny{$\bullet$}};
\node()at(5,-0.75){\tiny{$\bullet$}};
\node()at(7.25,0){\tiny{$\bullet$}};
\node()at(7.5,-0.5){\tiny{$\bullet$}};
\node()at(6.75,-0.25){\tiny{$\bullet$}};
\node()at(7,-0.75){\tiny{$\bullet$}};
\draw[dotted](5.25,0)--(5.5,-0.5);
\node()at(7.5,-1.75){\tiny{$\bullet$}};
\node()at(7.25,-1.25){\tiny{$\bullet$}};
\node()at(6.5,0.25){\tiny{$\bullet$}};
\node()at(6.25,0.75){\tiny{$\bullet$}};
\draw[-](7.5,-1.75)--(7.25,-1.25);
\draw[-](7.25,-1.25)--(7,-0.75);
\draw[-](6.5,0.25)--(6.75,-0.25);
\draw[-](6.5,0.25)--(6.25,0.75);
\node()at(3.75,1.75){\tiny{$\bullet$}};
\node()at(3.5,2.25){\tiny{$\bullet$}};
\node()at(6.25,2){\tiny{$\bullet$}};
\node()at(6,2.5){\tiny{$\bullet$}};
\draw[dotted](7,-0.75)--(6.75,-0.25);
\draw[-](7.5,-1.75)--(8,-1.5);
\draw[-](7.25,-1.25)--(7.75,-1);
\draw[-](6.5,0.25)--(7,0.5);
\draw[-](6.75,1)--(6.25,0.75);
\draw[-](7,-0.75)--(7.5,-0.5);
\draw[-](6.75,-0.25)--(7.25,0);
\node()at(8,-1.5){\tiny{$\bullet$}};
\node()at(7.75,-1){\tiny{$\bullet$}};
\node()at(7,0.5){\tiny{$\bullet$}};
\node()at(6.75,1){\tiny{$\bullet$}};
\draw[-](8,-1.5)--(7.75,-1);
\draw[-](7.75,-1)--(7.5,-0.5);
\draw[-](7,0.5)--(7.25,0);
\draw[-](7,0.5)--(6.75,1);
\draw[dotted](7.25,0)--(7.5,-0.5);
\node()at(6.5,1.5){\tiny{$\bullet$}};
\node()at(5.75,3){\tiny{$\bullet$}};
\draw[-](6.5,1.5)--(6.25,2);
\draw[-](6.5,1.5)--(6.75,1);
\draw[-](5.75,3)--(6,2.5);
\draw[dotted](6.25,2)--(6,2.5);
\end{tikzpicture}
\caption{Four types of Hasse diagrams}\label{fourHasse}
\end{figure}
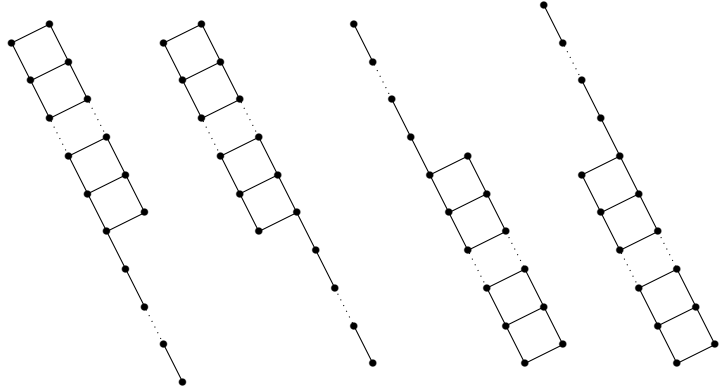

Let $X$ be a poset. The \emph{incidence algebra} $I(X)$ of $X$ with coefficients in ${\bf k}$ is the algebra with ${\bf k}$-basis $\{e_{xx'}|x\leq x'\in X\},$ whose multiplication is defined by
$$e_{xx'}e_{x'x''}=e_{xx''}$$
for $x\leq x'\leq x''\in X$ (all other products are zero).

The poset $X$ carries a structure of a topological space by defining the closed sets to be the subsets $Y\subseteq X$ satisfying that: if $y\in Y$ and $y'\leq y$, then $y'\in Y$.

Let $Y\subseteq X$ be a closed subset and $U$ be its complement. Then $X$ is the disjoint union of $Y$ and $U$. Let $A_Y$ be the algebra with the elements $\{e_{yy'}|y\leq y'\}\cup\{e_{u'u}|u'\leq u\}\cup\{e_{uy}|y<u\}$ as a $\bf{k}$-basis, where $y,y'\in Y$, $u,u'\in U$, and the multiplication is defined by
$$\begin{array}{lcl}e_{yy'}e_{y'y''}=e_{yy''}&&e_{u''u'}e_{u'u}=e_{u''u}\\
e_{uy}e_{yy'}=\left\{\begin{array}{lcl}e_{uy'}&&y'<u\\
0&&\textup{otherwise}\end{array}\right.&&e_{u'u}e_{uy}=\left\{\begin{array}{lcl}e_{u'y}&&y<u'\\
0&&\textup{otherwise}\end{array}\right.
\end{array}$$
for $y\leq y'\leq y''\in Y$, $u''\leq u'\leq u\in U$ (all other products are zero).

The following powerful result due to Ladkani plays an important role for the derived equivalences between incidence algebras and other algebras.

\begin{prop}\cite[Corollary 4.6]{Lad2008}\label{Ladkani} Keep notations as above. Then we have
$$\textup{D}^b(I(X))\simeq\textup{D}^b(A_Y).$$
\end{prop}

\subsection{Quiver representations and path algebras}

Let $A$ be a basic and connected finite dimensional ${\bf k}$-algebra. By Gabriel's theorem, there exists a finite connected quiver $Q$ and an admissible ideal $\rho$ such that $A={\bf k}Q/(\rho)$. A quiver $Q$ with an admissible ideal $\rho$ is called a bound quiver, and just denoted by $Q$ without confusion. We say that $A$ is the algebra associated to the bound quiver $Q$. It is obvious that the bound quivers in Figure \ref{algebraquiver} correspond to the incidence algebras of posets in Figure \ref{fourHasse}. Denoted by $A(u)$, $A(u)_v$, $A(u)^v$, $_vA(u)$, $^vA(u)$ the associated path algebras of the bound quivers $Q(u)$, $Q(u)_v$, $Q(u)^v$, $_vQ(u)$ and $^vQ(u)$, respectively.

Let $\textup{mod}\mbox{-}A$ be the category of finitely generated left $A$-modules. We choose a complete set $P_1,\cdots,P_n$ of representatives from the isomorphism classes of indecomposable projective left $A$-modules. The \emph{Cartan matrix} $C=C_{A}$ is defined to be the integer-valued $n\times n$-matrix with entries
$$c_{ij}=\textup{dim}_{\bf k}\textup{Hom}_A(P_i,P_j)\;\;\textup{for}\;1\leq i,j\leq n.$$
The entries of the Cartan matrix can also be interpreted as the dimensions of a two-sided Pierce decomposition of $A$. For this let $e_1,\cdots,e_n$ be a complete system of primitive orthogonal idempotents of $A$ such that $P_i=Ae_i$. Let $A=\oplus_{i,j}e_iAe_j$ be the two-sided Pierce decomposition with respect to $e_1,\cdots,e_n$. Then $c_{ij}=\textup{dim}_{\bf k}e_iAe_j$.

Let $K_0(A)$ denote the Grothendieck group of $A$, then $K_0(A)\simeq\mathbb{Z}^n$. We consider the elements $x=(x_1,\cdots,x_n)\in\mathbb{Z}^n$ as row vectors. The transpose of a vector $x\in\mathbb{Z}^n$ or a matrix $M$ is denoted by $x^t$ or $M^t$. If $\textup{gl.dim.}\;A<\infty$, then the Cartan matrix is invertible over $\mathbb{Z}$. Thus we can define a bilinear form $\left\langle-,-\right\rangle_A$ on $\mathbb{Z}^n$ by $\left\langle x,y\right\rangle_A=xC_A^{-t}y^t$ for $x,y\in\mathbb{Z}^n$. This bilinear form corresponds to the Euler form on $K_0(A)$, which has the following homological interpretation for $X,Y\in\textup{mod}\mbox{-}A$: $$\left\langle[X],[Y]\right\rangle_A=\sum_{i\geq0}\textup{dim}_{\bf k}\textup{Ext}_A^i(X,Y).$$

We denote by $\textup{D}^b(A)$ the bounded derived category of $\textup{mod}\mbox{-}A$. Then we can identify $K_0(A)$ with the Grothendieck group of the triangulated category $\textup{D}^b(A)$, see \cite{Gro1977,Hap1988}. In fact, if $X^{\bullet}=(X^i,d^i)\in\textup{D}^b(A)$ is a bounded complex, then the corresponding element $[X^{\bullet}]$ in $K_0(A)$ is $\sum_i(-1)^i[X^i]$. Using this identification we consider $\left\langle-,-\right\rangle_A$ also as a bilinear form on the Grothendieck group of $\textup{D}^b(A)$.

The \emph{Coxeter matrix} of $A$ is defined to be the matrix $$\Phi_{A}:=-C_{A}^{-t}C_A.$$ We denote by $\chi_A(\lambda)$ the characteristic polynomial of $\Phi_A$ and call it the \emph{Coxeter polynomial} of $A$.
Since $A$ is of finite global dimension, there is a triangle equivalence $\tau=\tau_{\textup{D}^b(A)}$ on $\textup{D}^b(A)$. It follows from the construction in \cite{Hap1988} that the linear transformation induced by $\tau$ on $K_0(A)$ coincides with $\Phi_A$. Thus for a $A$-module $X$ we have in $K_0(A)$ that $[\tau^j X]=x(\Phi_A)^j$ for all $j\in\mathbb{Z}$, where $x=\textup{dim} X$ is the dimension vector of $X$.

\subsection{Tilting objects}

Suppose $\mathcal{A}$ is an abelian category. Denote by $\textup{C}\mathcal{A}$ the category of complexes over $\mathcal{A}$ and by $\textup{H}\mathcal{A}$ the homotopy category. A quasi-isomorphism in $\textup{H}\mathcal{A}$ is a morphism whose image under the homology functors is invertible. Let $\Sigma$ be the class of quasi-isomorphisms in $\textup{H}\mathcal{A}$. We obtain the derived category $\textup{D}\mathcal{A}$\cite{Ver1977}: it is the category of fractions of the homotopy category with respect to the class of quasi-isomorphisms. A complex $I$ (resp. $P$) is \emph{fibrant}(resp. \emph{cofibrant}) if the canonical map
$$\textup{Hom}_{\textup{H}\mathcal{A}}(L,I)\rightarrow\textup{Hom}_{\textup{D}\mathcal{A}}(L,I)\;\textup{resp.}
\;\textup{Hom}_{\textup{H}\mathcal{A}}(P,L)\rightarrow\textup{Hom}_{\textup{D}\mathcal{A}}(P,L)$$
is bijective for each complex $L$.

\begin{lem}\label{fibrant}\cite[Lemma 2.3]{Kel2007}
If $I$ is a left bounded complex with injective components, then $I$ is fibrant.
Dually, if $P$ is a right bounded complex with projective components, then $P$ is cofibrant.
\end{lem}

Let $\mathcal{T}$ be a triangulated ${\bf k}$-category and $[1]$ be the corresponding translation functor. An indecomposable object $E$ in $\mathcal{T}$ is called \emph{exceptional} if $\textup{Hom}_{\mathcal{T}}(E,E)\cong\bf{k}$ and $\textup{Hom}_{\mathcal{T}}(E,E[i])=0$ for any $i\neq0$. A sequence $(E_1,\cdots,E_n)$ of exceptional objects in $\mathcal{T}$ is called an \emph{exceptional sequence} if $\textup{Hom}_{\mathcal{T}}(E_t,E_s[i])=0$ for any $1\leq s< t \leq n$ and $i\in\mathbb{Z}$; and it is further called a \emph{strongly exceptional sequence} if additionally $\textup{Hom}_{\mathcal{T}}(E_s,E_t[i])=0$ for any $1\leq s< t \leq n$ and $i\neq 0$.

\begin{lem}\label{tilt+equiv}\cite[Section 6]{Bon1989}
Let $\mathcal{A}$ be an abelian ${\bf k}$-category and let $(E_1,\cdots,E_n)$ be a strongly exceptional sequence which generates $\textup{D}^{b}(\mathcal{A})$. Set $T=\oplus_{s=1}^n E_s$. Then $T$ is a tilting object in $\textup{D}^{b}(\mathcal{A})$ and the functor
$$\textup{{\bf R}Hom}(T,-):\;\textup{D}^{b}(\mathcal{A})\rightarrow\textup{D}^{b}
(\textup{End}_{\textup{D}^{b}(\mathcal{A})}(T))$$
is a triangulated equivalence.
\end{lem}

\subsection{Weighted projective lines and the category of coherent sheaves}

In this paper, we only consider the weighted projective line of three weights. Let $p_1,p_2,p_3\geq2$ be integers, called weights. Denote by $S$ the commutative algebra
$$S=\frac{{\bf k}[X_1,X_2,X_3]}{(X_1^{p_1}+X_2^{p_2}+X_3^{p_3})}={\bf k}[x_1,x_2,x_3].$$
Let $\mathbb{L}=\mathbb{L}(p_1,p_2,p_3)$ be the abelian group given by generators $\vec{x}_1,\vec{x}_2,\vec{x}_3$ and defining relations $p_1\vec{x}_1=p_2\vec{x}_2=p_3\vec{x}_3:=\vec{c}$. The $\mathbb{L}$-graded algebra $S$ is the appropriate object to study the triangle singularity $x_1^{p_1}+x_2^{p_2}+x_3^{p_3}$. The element $\vec{c}$ is called the \emph{canonical element}. Each element $\vec{x}\in\mathbb{L}$ can be written in \emph{normal form} $\vec{x}=n_1\vec{x}_1+n_2\vec{x}_2+n_3\vec{x}_3+m\vec{c}$ with unique $n_i,m\in\mathbb{Z}$, $0\leq n_i<p_i$. The algebra $S$ is $\mathbb{L}$-graded by setting $\textup{deg}\;x_i=\vec{x}_i\,(i=1,2,3)$.

Let $\mathbb{X}=\mathbb{X}(p_1,p_2,p_3)$ denote the weighted projective line of weight type $(p_1,p_2,p_3)$. By an $\mathbb{L}$-graded version of the Serre construction \cite{Ser1955}, $\mathbb{X}$ is given by its category of coherent sheaves $\textup{coh}\mbox{-}\mathbb{X}=\textup{mod}^{\mathbb{L}}\mbox{-}S/\textup{mod}_0^{\mathbb{L}}\mbox{-}S$, the quotient category of finitely generated $\mathbb{L}$-graded modules modulo the Serre subcategory of graded modules of finite length. The abelian group $\mathbb{L}$ is ordered by defining the positive cone $\{\vec{x}\in\mathbb{L}|\vec{x}\geq0\}$ to consist of the elements of the form $n_1\vec{x}_1+n_2\vec{x}_2+n_3\vec{x}_3$, where $n_1,n_2,n_3\geq0$.

The image $\mathcal{O}$ of $S$ in $\textup{mod}^{\mathbb{L}}\mbox{-}S/\textup{mod}_0^{\mathbb{L}}\mbox{-}S$ serves as the structure sheaf of $\textup{coh}\mbox{-}\mathbb{X}$, and $\mathbb{L}$ acts on the above data, in particular on $\textup{coh}\mbox{-}\mathbb{X}$, by grading shift. Each line bundle has the form $\mathcal{O}(\vec{x})$ for a uniquely determined $\vec{x}$ in $\mathbb{L}$. Defining the \emph{dualizing element} of $\mathbb{L}$ as $\vec{\omega}=\vec{c}-(\vec{x}_1+\vec{x}_2+\vec{x}_3)$, the category $\textup{coh}\mbox{-}\mathbb{X}$ satisfies Serre duality in the form $D\textup{Ext}^1(X,Y)=\textup{Hom}(Y,X(\vec{\omega}))$, functorially in $X$ and $Y$. Moreover, Serre duality implies the existence of almost split sequences for $\textup{coh}\mbox{-}\mathbb{X}$ with the Auslander-Reiten translation $\tau$ given by the shift with $\vec{\omega}$.

\subsection{Stable category of vector bundles}

Denote by $\textup{vect}\mbox{-}\mathbb{X}$ the full subcategory of the category $\textup{coh}\mbox{-}\mathbb{X}$ formed by all vector bundles. The category $\textup{vect}\mbox{-}\mathbb{X}$ carries the structure of a Frobenius category such that the system $\mathcal{L}$ of all line bundles is the system of all indecomposable projective-injectives: a sequence $\eta$: $0\rightarrow E'\rightarrow E\rightarrow E''\rightarrow 0$ in $\textup{vect}\mbox{-}\mathbb{X}$ is \emph{distinguished exact} if all the sequence $\textup{Hom}(L,\eta)$ with $L$ a line bundle are exact. Accordingly, the stable category $\textup{\underline{vect}}\mbox{-}\mathbb{X}=\textup{vect}\mbox{-}\mathbb{X}/[\mathcal{L}]$ is triangulated by \cite{Hap1988}. Moreover, it follows that the triangulated category $\textup{\underline{vect}}\mbox{-}\mathbb{X}$ is Krull-Schmidt with Serre duality $\mathbb{S}$ induced from that of $\textup{coh}\mbox{-}\mathbb{X}$.

For each line bundle $L$, the extension term of the almost split sequence
$$0\rightarrow L(\vec{\omega})\rightarrow E_L\rightarrow L\rightarrow 0$$
is called the \emph{Auslander bundle} corresponding to $L$. More generally, we consider \emph{extension bundles} $E_L\langle\vec{x}\rangle$ defined as the extension terms of the `unique' non-split exact sequence
$$0\rightarrow L(\vec{\omega})\rightarrow E_L\langle\vec{x}\rangle\rightarrow L(\vec{x})\rightarrow 0,$$
where $0\leq\vec{x}\leq2\vec{\omega}+\vec{c}$. For the sake of simplicity, we denote $E_{\mathcal{O}}\langle\vec{x}\rangle$ by $E\langle\vec{x}\rangle$. Since we are dealing with the case of three weights $p_i\geq2\,(i=1,2,3)$, each extension bundle, in particular each Auslander bundle, is exceptional in $\textup{\underline{vect}}\mbox{-}\mathbb{X}$ by \cite[Corollary 4.11]{KLM2013}.

\begin{lem}\cite[Corollary 4.14]{KLM2013}\label{414}
Let $E$ be the Auslander bundle and $\vec{x}\in\mathbb{L}$. Then $\textup{Hom}_{\underline{\textup{vect}}\mbox{-}\mathbb{X}}(E,E(\vec{x}))\neq0$ if and only if $\vec{x}\in\{0,\vec{x}_1+\vec{\omega},\vec{x}_2+\vec{\omega},\vec{x}_3+\vec{\omega}\}$, and in this case $\textup{Hom}_{\underline{\textup{vect}}\mbox{-}\mathbb{X}}(E,E(\vec{x}))$ is isomorphic to ${\bf k}$.
\end{lem}
\subsection{One-point extension}

Let $M$ be a left $A$-module. The \emph{one-point extension} $A[M]$ of $A$ by $M$ is defined as the triangular matrix ring
$$A[M]=\begin{bmatrix}A&M\\0&{\bf k}\end{bmatrix}$$
with multiplication given by
$$\begin{pmatrix}a&m\\0&\alpha\end{pmatrix}\begin{pmatrix}a'&m'\\0&\alpha'\end{pmatrix}=\begin{pmatrix}aa'& am'+m\alpha\\0&\alpha\alpha'\end{pmatrix}$$
for $a,a'\in A$, $m,m'\in M$ and $\alpha,\alpha'\in {\bf k}$.

If $A$ has finite global dimension, we trivially have that $A[M]$ has finite global dimension. In fact, $\textup{gl.dim}\;A[M]\leq\textup{gl.dim}\;A+1$.

For one-point extension algebra, we have following useful results.

\begin{prop}\cite[Theorem 3.2]{Hap2009}\label{Hap 1-ex}
Let $A$ be a finite-dimensional algebra of finite global dimension and let $M\in\textup{mod}\mbox{-}A$. Let $B=A[M]$ be the one-point extension algebra. Let $\chi_{A}(\lambda)=\sum_{i=0}^n a_i\lambda^{n-i}$ and $\chi_{B}(\lambda)=\sum_{i=0}^{n+1}b_i\lambda^{n+1-i}$ be the Coxeter polynomials. Then
$$b_i=a_i-a_{i-1}(\left\langle M,M\right\rangle_A-1)-\sum_{j=2}^i a_{i-j}\left\langle \tau^{j-1}M,M\right\rangle_A,$$
where we set $a_{-1}=a_{n+1}=0$.
\end{prop}

\begin{prop}\cite[Proposition 18.3 and Corollary 18.2]{Len1997}\label{one-point-extension}
Let $M$ be an exceptional $A$-module and $B=A[M]$ be an one-point extension of $A$ by $M$. Let $M^{\perp}$ be the right perpendicular category. Assume that $M^{\perp}$ is derived equivalent to $\textup{mod}\mbox{-}A'$ for some algebra $A'$.
Then $$\chi_{B}(\lambda)=(1+\lambda)\cdot\chi_A(\lambda)-\lambda\cdot \chi_{A'}(\lambda).$$
\end{prop}

\section{Coxeter polynomial of $A(u)_v$}

In this section, we will give the formulas of the Coxeter polynomials for the one-branch extension algebra $A(u)_v$ for any $u\geq 1, v\geq 1$.

\subsection{Coxeter polynomial of $A(u)$}

In \cite{HM2014}, Hille-M\"uller have calculated the Coxeter polynomials for the tensor product of path algebras of Dynkin type $\mathbb{A}$. In particular, for $A(u)$ we have:

\begin{prop}\cite[Example 2.8(a)]{HM2014}\label{d=0}
For any $u\geq 1$, the Coxeter polynomial of $A(u)$ is given by
$$\chi_{A(u)}(\lambda)=\left\{\begin{array}{lcl}\frac{(\lambda+1)\big(\lambda^{3u+3}-(-1)^{u+1}\big)}
{(\lambda^3+1)\big(\lambda^{u+1}-(-1)^{u+1}\big)},&&u\not\equiv 2\,(\mod 3),\\\\
\frac{(\lambda+1)\big(\lambda^{u+1}-(-1)^{u+1}\big)^2}{\lambda^3+1},&&u\equiv 2\,(\mod 3).\end{array}\right.$$
\end{prop}

For later use, we need to expand the Coxeter polynomials $\chi_{A(u)}(\lambda)$.
Observe that for any integer $k\geq 1$,
$$\frac{\lambda^{3k}-(-1)^{k}}{\lambda^k-(-1)^{k}}
=\frac{(\lambda^{k})^3-((-1)^{k})^3}{\lambda^k-(-1)^{k}}
=\lambda^{2k}+(-\lambda)^k+1;$$
and
$$\frac{\lambda^{3k}-(-1)^{3k}}{\lambda^3+1}
=\frac{(\lambda^{3})^k-(-1)^{k}}{\lambda^3-(-1)}
=\sum_{j=0}^{k-1}(-1)^{k-1-j}\lambda^{3j}
=(-1)^{k-1}\sum_{j=0}^{k-1}(-\lambda)^{3j}.$$
Therefore, for $u=3i$,
\begin{align}\notag\chi_{A(u)}(\lambda)&=\frac{(\lambda+1)(\lambda^{3u+3}-(-1)^{u+1})}
{(\lambda^3+1)(\lambda^{u+1}-(-1)^{u+1})}\\\notag
&=\frac{\lambda+1}{\lambda^3+1}(\lambda^{2u+2}+(-\lambda)^{u+1}+1)\\\notag
&=\frac{\lambda+1}{\lambda^3+1}[(\lambda^{2u+2}-(-\lambda)^{u+2})+((-\lambda)^{u+2}
+(-\lambda)^{u+1}+(-\lambda)^{u})-((-\lambda)^{u}-1)]\\\notag
&=(\lambda+1)\lambda^{u+2}\cdot\frac{\lambda^{u}-(-1)^{u}}{\lambda^3+1}+(-\lambda)^{u}
-(\lambda+1)\cdot(-1)^u\cdot\frac{\lambda^{u}-(-1)^{u}}{\lambda^3+1}\\\notag
&=(\lambda+1)\lambda^{3i+2}\cdot(-1)^{i-1}\sum_{j=0}^{i-1}(-\lambda)^{3j}+(-\lambda)^{3i}
-(\lambda+1)\cdot(-1)^{3i}\cdot(-1)^{i-1}\sum_{j=0}^{i-1}(-\lambda)^{3j}\\\label{remark1}
&=-(\lambda+1)\cdot\sum_{j=0}^{i-1}(-\lambda)^{3i+2+3j}+(-\lambda)^{3i}+
(\lambda+1)\cdot\sum_{j=0}^{i-1}(-\lambda)^{3j}.
\end{align}
For $u=3i+1$,
\begin{align}\notag\chi_{A(u)}(\lambda)&
=\frac{\lambda+1}{\lambda^3+1}(\lambda^{2u+2}+(-\lambda)^{u+1}+1)\\\notag
&=\frac{\lambda+1}{\lambda^3+1}[(\lambda^{2u+2}-(-\lambda)^{u+3})+((-\lambda)^{u+3}
+(-\lambda)^{u+1}+(-\lambda)^{u-1})-((-\lambda)^{u-1}-1)]\\\notag
&=(\lambda+1)\lambda^{u+3}\cdot\frac{\lambda^{u-1}-(-1)^{u-1}}{\lambda^3+1}
+(-\lambda)^{u-1}\cdot\frac{(\lambda+1)(\lambda^{4}+\lambda^2+1)}{\lambda^3+1}\\\notag
&\qquad-(\lambda+1)\cdot(-1)^{u-1}\cdot\frac{\lambda^{u-1}-(-1)^{u-1}}{\lambda^3+1}\\\notag
&=(\lambda+1)\lambda^{3i+4}\cdot(-1)^{i-1}\sum_{j=0}^{i-1}(-\lambda)^{3j}
+(-\lambda)^{3i}\cdot(\lambda^{2}+\lambda+1)\\\notag
&\qquad-(\lambda+1)\cdot(-1)^{3i}\cdot(-1)^{i-1}\sum_{j=0}^{i-1}(-\lambda)^{3j}\\\label{remark2}
&=-(\lambda+1)\cdot\sum_{j=0}^{i-1}(-\lambda)^{3i+4+3j}+(-\lambda)^{3i}\cdot(\lambda^{2}+\lambda+1)
+(\lambda+1)\cdot\sum_{j=0}^{i-1}(-\lambda)^{3j}.
\end{align}
For $u=3i+2$,
\begin{align}\notag\chi_{A(u)}(\lambda)&=\frac{(\lambda+1)(\lambda^{u+1}-(-1)^{u+1})^2}
{\lambda^3+1}\\\notag
&=(\lambda+1)(\lambda^{3i+3}-(-1)^{3i+3})\frac{\lambda^{3i+3}-(-1)^{3i+3}}{\lambda^3+1}\\\notag
&=(\lambda+1)(\lambda^{3i+3}-(-1)^{3i+3})\cdot (-1)^{i}\sum_{j=0}^i(-\lambda)^{3j}\\\label{remark3}
&=-(\lambda+1)\cdot\sum_{j=0}^{i}(-\lambda)^{3i+3+3j}+(\lambda+1)\cdot\sum_{j=0}^{i}(-\lambda)^{3j}.
\end{align}

\subsection{Euler form between Auslander bundles}

Recall that the path algebra $A(u)$ has a close relationship with the weighted projective line of weight type $(2,3,u+1)$. From now on, we always assume that $\mathbb{X}$ is a weighted projective line of weight type $(2,3,u+1)$. Let $\mathbb{L}(2,3,u+1)$ be the abelian group given by generators $\vec{x}_1,\vec{x}_2,\vec{x}_3$ and defining relations $2\vec{x}_1=3\vec{x}_2=(u+1)\vec{x}_3:=\vec{c}$.

\begin{lem}\label{lemma<>}
Assume $a,b\in\mathbb{Z}$ and $1\leq b\leq 2u+1$. Then in $\mathbb{L}(2,3,u+1)$ we have
\begin{itemize}
\item[(1)] $a\vec{x}_1-b\vec{\omega}=0$ if and only if $3|(u+1)$, $a=(u-5)/3$, $b=u+1$;
\item[(2)] $a\vec{x}_1-b\vec{\omega}=\vec{x}_1$ if and only if $3|(u+1)$, $a=(u-2)/3$, $b=u+1$;
\item[(3)] $a\vec{x}_1-b\vec{\omega}=\vec{x}_2$ if and only if $3|u$, $a=(u-3)/3$, $b=u+1$;
\item[(4)] $a\vec{x}_1-b\vec{\omega}=\vec{x}_3$ if and only if $3|(u+2)$, $a=(u-4)/3$, $b=u+2$.
\end{itemize}
\end{lem}

\begin{proof}
(1) If $a\vec{x}_1-b\vec{\omega}=0$, we have $(a+b)\vec{x}_1+b\vec{x}_2+b\vec{x}_3-b\vec{c}=0$. It implies that $2|(b+a)$, $3|b$, $(u+1)|b$ and $b=\frac{b+a}{2}+\frac{b}{3}+\frac{b}{u+1}$. That is, $$\frac{a}{2b}=\frac{1}{2}-\frac{1}{3}-\frac{1}{u+1}=\frac{1}{6}-\frac{1}{u+1}=\frac{u-5}{6(u+1)}. $$ If $u=5$, we get $a=0$ and $6|b$. It implies $b=6$ since $1\leq b\leq 2u+1$. If $u\neq5$, then $$b=\frac{3a(u+1)}{u-5}=3a+\frac{18a}{u-5}.$$ Note that $2|(b+a)$ and $3|b$. It follows that $6|(b-3a)$, that is, $6|\frac{18a}{u-5}$, which implies $(u-5)|(3a)$. Assume $3a=t(u-5)$ for some integer $t$. Then we have $b=t(u+1)$. Recall that $1\leq b\leq 2u+1$. So $t=1$ and then $b=u+1$, $a=\frac{u-5}{3}$. The converse is trivial.

(2) Note that $a\vec{x}_1-b\vec{\omega}=\vec{x}_1$ if and only if $(a-1)\vec{x}_1-b\vec{\omega}=0$, then the result follows from (1).

(3) If $a\vec{x}_1-b\vec{\omega}=\vec{x}_2$, we have $(a+b)\vec{x}_1+(b-1)\vec{x}_2+b\vec{x}_3-b\vec{c}=0$. It implies that $2|(b+a)$, $3|(b-1)$, $(u+1)|b$ and $b=\frac{b+a}{2}+\frac{b-1}{3}+\frac{b}{u+1}$. That is, $$\frac{3a-2}{6b}=\frac{u-5}{6(u+1)}.$$ Then $a\in\mathbb{Z}$ implies $u\neq5$. Hence $$b=\frac{(3a-2)(u+1)}{u-5}=3a-2+\frac{6(3a-2)}{u-5}.$$ Note that $3|(b-1)$ and $2|(b+a)$. It follows that $6|(b-3a+2)$, that is, $6|\frac{6(3a-2)}{u-5}$. It implies $(u-5)|(3a-2)$. Assume $3a-2=t(u-5)$ for some integer $t$. Then we have $b=t(u+1)$. Recall that $1\leq b\leq 2u+1$. So $t=1$ and then $b=u+1$, $a=\frac{u-3}{3}$.

(4) If $a\vec{x}_1-b\vec{\omega}=\vec{x}_3$, we have $(a+b)\vec{x}_1+b\vec{x}_2+(b-1)\vec{x}_3-b\vec{c}=0$. It implies that $2|(b+a)$, $3|b$, $(u+1)|(b-1)$ and $b=\frac{b+a}{2}+\frac{b}{3}+\frac{b-1}{u+1}$. That is, $$\frac{a(u+1)-2}{2b(u+1)}=\frac{u-5}{6(u+1)}.$$ Clearly, $u\neq5$ and then $$b=\frac{3(a(u+1)-2)}{u-5}=3a+\frac{6(3a-1)}{u-5}.$$ Note that $3|b$ and $2|(b+a)$. It follows that $6|(b-3a)$, that is, $6|\frac{6(3a-1)}{u-5}$. It implies $(u-5)|(3a-1)$. Assume $3a-1=t(u-5)$ for some integer $t$. Then we have $b=t(u+1)+1$. Recall that $1\leq b\leq 2u+1$. So $t=1$ and then  $b=u+2$, $a=\frac{u-4}{3}$.
\end{proof}

\begin{lem}\label{Euler form of tau E and E} Let $E$ be an Auslander bundle and $\mathbb{S}$ be the Serre functor of $\textup{\underline{vect}}\mbox{-}\mathbb{X}$. Then for any $1\leq j\leq 2u+1$, $\left\langle \mathbb{S}^jE,E\right\rangle\neq 0$ if and only if one of the following holds:
$$(1)\, 3|n,\, j=u; \quad (2)\, 3|(u+1),\, j=u\; \textup{{or}} \;u+1; \quad (3)\, 3|(u+2),\, j=u+1. $$
\end{lem}

\begin{proof}
Since $[1]=(\vec{x}_1)$ and $[2]=(\vec{c})$ in $\textup{\underline{vect}}\mbox{-}\mathbb{X}$, we have
$$\begin{array}{ll}\left\langle \mathbb{S}^jE,E\right\rangle
&=\sum\limits_{l\in\mathbb{Z}}(-1)^l\textup{dim}\;\textup{Hom}_{\textup{\underline{vect}}\mbox{-}\mathbb{X}}(E,\mathbb{S}^{-j}E[l])\\
&=\sum\limits_{l\in\mathbb{Z}}(-1)^l\textup{dim}\;\textup{Hom}_{\textup{\underline{vect}}\mbox{-}\mathbb{X}}(E,E(l\vec{x}_1-j\vec{\omega})).\end{array}
$$
By Lemma \ref{414}, $\textup{Hom}_{\textup{\underline{vect}}\mbox{-}\mathbb{X}}(E,E(l\vec{x}_1-j\vec{\omega}))\neq 0$ if and only if $l\vec{x}_1-j\vec{\omega}\in\{0,\vec{x}_1+\vec{\omega},\vec{x}_2+\vec{\omega},\vec{x}_3+\vec{\omega}\}$. Then the result follows from Lemma \ref{lemma<>} directly.
\end{proof}

Let $A(u)$ be the path algebra associated to the bound quiver $Q(u)$ as follow:
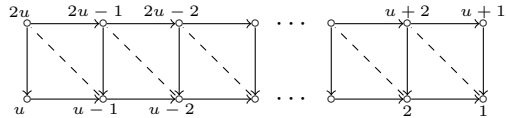
\begin{figure}[H]\begin{tikzpicture}
\node()at(0,0){\tiny{$\circ$}};
\node()at(-0.1,-0.15){\tiny{$u$}};
\node()at(-0.1,1.15){\tiny{$2u$}};
\node()at(0.9,-0.15){\tiny{$u-1$}};
\node()at(0.9,1.15){\tiny{$2u-1$}};
\node()at(1.9,-0.15){\tiny{$u-2$}};
\node()at(1.9,1.15){\tiny{$2u-2$}};
\node()at(1,0){\tiny{$\circ$}};
\node()at(1,1){\tiny{$\circ$}};
\node()at(0,1){\tiny{$\circ$}};
\node()at(2,0){\tiny{$\circ$}};
\node()at(2,1){\tiny{$\circ$}};
\node()at(3,0){\tiny{$\circ$}};
\node()at(3,1){\tiny{$\circ$}};
\node()at(4,0){\tiny{$\circ$}};
\node()at(4,1){\tiny{$\circ$}};
\node()at(5,0){\tiny{$\circ$}};
\node()at(5,1){\tiny{$\circ$}};
\node()at(6,0){\tiny{$\circ$}};
\node()at(6,1){\tiny{$\circ$}};
\node()at(6,-0.15){\tiny{$1$}};
\node()at(6,1.15){\tiny{$u+1$}};
\node()at(5,-0.15){\tiny{$2$}};
\node()at(5,1.15){\tiny{$u+2$}};
\node()at(3.5,0){$\cdots$};
\node()at(3.5,1){$\cdots$};
\draw[-,dashed](0.95,0.05)--(0.05,0.95);
\draw[-,dashed](1.95,0.05)--(1.05,0.95);
\draw[-,dashed](2.95,0.05)--(2.05,0.95);
\draw[-,dashed](4.95,0.05)--(4.05,0.95);
\draw[-,dashed](5.95,0.05)--(5.05,0.95);
\draw[->](0,0.95)--(0,0.05);
\draw[->](1,0.95)--(1,0.05);
\draw[->](0.05,0)--(0.95,0);
\draw[->](0.05,1)--(0.95,1);
\draw[->](2,0.95)--(2,0.05);
\draw[->](3,0.95)--(3,0.05);
\draw[->](4,0.95)--(4,0.05);
\draw[->](1.05,0)--(1.95,0);
\draw[->](1.05,1)--(1.95,1);
\draw[->](2.05,0)--(2.95,0);
\draw[->](2.05,1)--(2.95,1);
\draw[->](4.05,0)--(4.95,0);
\draw[->](4.05,1)--(4.95,1);
\draw[->](5,0.95)--(5,0.05);
\draw[->](5.05,0)--(5.95,0);
\draw[->](5.05,1)--(5.95,1);
\draw[->](6,0.95)--(6,0.05);
\end{tikzpicture}
\caption{The bound quiver $Q(u)$}\label{figurerect}
\end{figure}
Let $M$ be the injective module corresponding to the vertex $1$. Then we have the following result.

\begin{lem}\label{<mf,m>} Let $1\leq j\leq 2u+1$. We have
\begin{itemize}
\item[(1)] for $u\equiv 0\,(\mod 3)$, $\left\langle \tau^jM,M\right\rangle=\left\{\begin{array}{lcl}(-1)^{(u-3)/3},&&j=u,\\0,&&\textup{else;}
\end{array}\right.$
\item[(2)] for $u\equiv 1\,(\mod 3)$, $\left\langle \tau^jM,M\right\rangle=\left\{\begin{array}{lcl}(-1)^{(u-4)/3},&&j=u+1,\\0,&&\textup{else;}
\end{array}\right.$
\item[(3)] for $u\equiv 2\,(\mod 3)$, $\left\langle \tau^jM,M\right\rangle=\left\{\begin{array}{lcl}(-1)^{(u-2)/3},&&j=u,\\(-1)^{(u-5)/3},&&j=u+1,\\0,&&\textup{else.}
\end{array}\right.$
\end{itemize}
\end{lem}

\begin{proof}
We consider the category $\textup{\underline{vect}}\mbox{-}\mathbb{X}$ for $\mathbb{X}$ of weight type $(2,3,u+1)$. By \cite[Theorem 6.1]{KLM2013},
$$T_{\textup{cub}}=(\bigoplus_{i=0}^{u-1}E\langle i\vec{x}_3\rangle) \oplus (\bigoplus_{i=0}^{u-1} E\langle\vec{x}_2+i\vec{x}_3\rangle)$$
is tilting objects in $\textup{\underline{vect}}\mbox{-}\mathbb{X}$, whose endomorphism algebra has the following shape:
$$\tiny{\xymatrix{E\ar[r]\ar[d]&E\langle\vec{x}_3\rangle\ar[r]\ar[d]&\cdots\ar[r]&E\langle(u-2)\vec{x}_3\rangle\ar[r]\ar[d]&E\langle(u-1)\vec{x}_3\rangle\ar[d]\\
E\langle{\vec{x}_2}\rangle\ar[r]&E\langle\vec{x}_2+\vec{x}_3\rangle\ar[r]&\cdots\ar[r]&E\langle\vec{x}_2+(u-2)\vec{x}_3\rangle\ar[r]&E\langle\vec{x}_2+(u-1)\vec{x}_3\rangle.}}$$
We also note that by \cite[Proposition 4.15]{KLM2013}, we have $$E\langle\vec{x}_2+(u-1)\vec{x}_3\rangle=E(\vec{c}-\vec{x}_2-\vec{x}_3)=\mathbb{S}E.$$
Therefore, $$\begin{array}{ll}\left\langle \tau^jM,M\right\rangle&=\left\langle \mathbb{S}^jE\langle\vec{x}_2+(u-1)\vec{x}_3\rangle,E\langle\vec{x}_2+(u-1)\vec{x}_3\rangle\right\rangle=\left\langle\mathbb{S}^{j}E,E\right\rangle.\end{array}$$
Recall that the nonzero homomorphism between Auslander bundles are of dimension one, see Lemma \ref{414}. Then the result follows from Lemma \ref{Euler form of tau E and E}.
\end{proof}

\subsection{Coxeter polynomial of $A(u)_1$}

Now we can give the formula of the Coxeter polynomial of the one-point extension algebra $A(u)_1$ associated to $Q(u)_1$.

\begin{prop}\label{d=1}
The Coxeter polynomial of $A(u)_1$ is given by
$$\chi_{A(u)_1}(\lambda)=\left\{\begin{array}{ll}
\frac{(\lambda+1)(\lambda^{2u+3}+1)}{\lambda^3+1}, &u=3i,\\\\
\frac{(\lambda+1)\big(\lambda^{u+1}+(-1)^{u+1}\big)\big(\lambda^{u+2}+(-1)^{u+1}\big)}{\lambda^3+1}, &u=3i+1,\\\\
\frac{(\lambda+1)\big(\lambda^{u+1}-(-1)^{u+1}\big)\big(\lambda^{u+2}-(-1)^{u+1}\big)}{\lambda^3+1},&u=3i+2.\end{array}\right.$$
\end{prop}

\begin{proof}
Let $A(u)[M]$ be the one-point extension algebra of $A(u)$ by the injective module $M$ corresponding to the vertex $1$, c.f. Figure \ref{figurerect}.
Observe that the path algebra $A(u)_1$ associated to the bound quiver $Q(u)_1$ is isomorphic to $A(u)[M]$.

Let $$\chi_{A(u)}(\lambda)=\sum_{j=0}^{2u} a_j\lambda^{2u-j}\quad\text{and}\quad\chi_{A(u)_1}(\lambda)=\sum_{j=0}^{2u+1}b_j\lambda^{2u+1-j}$$ be the Coxeter polynomials of $A(u)$ and $A(u)_1$ respectively. Since $M$ is exceptional, we have $\left\langle M,M\right\rangle=1$. Then by Proposition \ref{Hap 1-ex},
$$b_j=a_j-\sum_{l=2}^j a_{j-l}\left\langle \tau^{l-1}M,M\right\rangle.$$

(1) For $u=3i$, by Lemma \ref{<mf,m>} (1), we have $\left\langle \tau^{j-1}M,M\right\rangle\neq0$ if and only if $j=u+1$, and $\left\langle \tau^{u}M,M\right\rangle=(-1)^{i-1}$. Hence,
$$b_j=\left\{\begin{array}{ll}a_j,&0\leq j\leq u,\\
a_j-(-1)^{i-1}a_{j-u-1},&u+1\leq j\leq 2u+1.\end{array}\right.$$
Hence,
\begin{align*}\chi_{A(u)_1}(\lambda)&=\sum_{j=0}^{2u+1}b_j\lambda^{2u+1-j}\\
&=\sum_{j=0}^{u}a_j\lambda^{2u+1-j}+\sum_{j=u+1}^{2u+1}(a_j-(-1)^{i-1}a_{j-u-1})\lambda^{2u+1-j}\\
&=\sum_{j=0}^{2u+1}a_j\lambda^{2u+1-j}-(-1)^{i-1}\sum_{j=u+1}^{2u+1}a_{j-u-1}\lambda^{2u+1-j}\\
&=\lambda\cdot\chi_{A(u)}(\lambda)+(-1)^i\sum_{j=0}^{u}a_j\lambda^{u-j}.\end{align*}
By formula \ref{remark1}, we have that
\begin{align*}\sum_{j=0}^{u}a_j\lambda^{u-j}&=\frac{1}{\lambda^{u}}\sum_{j=0}^{u}a_j\lambda^{2u-j}\\
&=\frac{1}{\lambda^{u}}(-(\lambda+1)\sum_{j=0}^{i-1}(-\lambda)^{3i+2+3j}+(-\lambda)^{3i})\\
&=(-1)^{u+1}(\lambda+1)\sum_{j=0}^{i-1}(-\lambda)^{2+3j}+(-1)^{u}\\
&=(-1)^{u+1}(\lambda+1)\frac{\lambda^2-(-\lambda)^{2+3i}}{1-(-\lambda)^{3}}+(-1)^{u}\\
&=(-1)^{u+1}\frac{(\lambda+1)(\lambda^2-(-\lambda)^{2+3i})-(\lambda^3+1)}{\lambda^3+1}\\
&=(-1)^{u+1}(\lambda+1)\frac{\lambda^2-(-\lambda)^{2+3i}-\lambda^2+\lambda-1}{\lambda^3+1}\\
&=(-1)^{u}\frac{\lambda+1}{\lambda^3+1}((-\lambda)^{2+3i}-\lambda+1).
\end{align*}
Therefore,
\begin{align*}\chi_{A(u)_1}(\lambda)&=\lambda\cdot\chi_{A(u)}(\lambda)+\frac{\lambda+1}{\lambda^3+1}((-\lambda)^{2+3i}-\lambda+1)\\
&=\frac{\lambda+1}{\lambda^3+1}\big(\lambda(\lambda^{2u+2}+(-\lambda)^{u+1}+1)
+(-\lambda)^{2+u}-\lambda+1\big)\\
&=\frac{\lambda+1}{\lambda^3+1}(\lambda^{2u+3}+1)
\end{align*}

(2) For $u=3i+1$, by Lemma \ref{<mf,m>} (2), we have $\left\langle \tau^{j-1}M,M\right\rangle\neq0$ if and only if $j=u+2$, and $\left\langle \tau^{u+1}M,M\right\rangle=(-1)^{i-1}$. Hence,
$$b_j=\left\{\begin{array}{ll}a_j,&0\leq j\leq u+1,\\
a_j-(-1)^{i-1}a_{j-u-2},&u+2\leq j\leq 2u+1.\end{array}\right.$$
Hence,
\begin{align*}\chi_{A(u)_1}(\lambda)&=\sum_{j=0}^{2u+1}b_j\lambda^{2u+1-j}\\
&=\sum_{j=0}^{u+1}a_j\lambda^{2u+1-j}+\sum_{j=u+2}^{2u+1}(a_j-(-1)^{i-1}a_{j-u-2})\lambda^{2u+1-j}\\
&=\sum_{j=0}^{2u+1}a_j\lambda^{2u+1-j}-(-1)^{i-1}\sum_{j=u+2}^{2u+1}a_{j-u-2}\lambda^{2u+1-j}\\
&=\lambda\cdot\chi_{A(u)}(\lambda)+(-1)^i\sum_{j=0}^{u-1}a_j\lambda^{u-1-j}\\
\end{align*}
By formula \ref{remark2}, we have that
\begin{align*}\sum_{j=0}^{u-1}a_j\lambda^{u-1-j}&=\frac{1}{\lambda^{u+1}}\sum_{j=0}^{u-1}a_j\lambda^{2u-j}\\
&=\frac{1}{\lambda^{u+1}}\big(-(\lambda+1)\cdot\sum_{j=0}^{i-1}(-\lambda)^{3i+4+3j}+(-\lambda)^{3i+2}\big)\\
&=(-1)^{u+2}(\lambda+1)\sum_{j=0}^{i-1}(-\lambda)^{2+3j}+(-1)^{u+1}\\
&=(-1)^{u}(\lambda+1)\frac{\lambda^2-(-\lambda)^{2+3i}}{1-(-\lambda)^{3}}+(-1)^{u+1}\\
&=(-1)^{u}\frac{(\lambda+1)(\lambda^2-(-\lambda)^{2+3i})-(\lambda^3+1)}{\lambda^3+1}\\
&=(-1)^{u}(\lambda+1)\frac{\lambda^2-(-\lambda)^{2+3i}-\lambda^2+\lambda-1}{\lambda^3+1}\\
&=(-1)^{u+1}\frac{\lambda+1}{\lambda^3+1}((-\lambda)^{2+3i}-\lambda+1).
\end{align*}
Therefore,
\begin{align*}\chi_{A(u)_1}(\lambda)&=\lambda\cdot\chi_{A(u)}(\lambda)
+\frac{\lambda+1}{\lambda^3+1}\cdot\big((-\lambda)^{2+3i}-\lambda+1\big)\\
&=\frac{\lambda+1}{\lambda^3+1}\cdot\big(\lambda(\lambda^{2u+2}+(-\lambda)^{u+1}+1)+((-\lambda)^{u+1}-\lambda+1)\big)\\
&=\frac{\lambda+1}{\lambda^3+1}\cdot(\lambda^{2u+3}-(-\lambda)^{u+2}+(-\lambda)^{u+1}+1)\\
&=\frac{(\lambda+1)(\lambda^{u+1}+(-1)^{u+1})(\lambda^{u+2}+(-1)^{u+1})}{\lambda^3+1}
\end{align*}

(3) For $u=3i+2$, by Lemma \ref{<mf,m>} (3), we have $\left\langle \tau^{j-1}M,M\right\rangle\neq0$ if and only if $j=u+1, u+2$, and $\left\langle \tau^{u}M,M\right\rangle=(-1)^{i}$, $\left\langle \tau^{u+1}M,M\right\rangle=(-1)^{i-1}$. Hence,
$$b_j=\left\{\begin{array}{ll}a_j,&0\leq j\leq u,\\
a_{u+1}-(-1)^ia_0,&j=u+1,\\
a_j-(-1)^{i}a_{j-u-1}-(-1)^{i-1}a_{j-u-2},&u+2\leq j\leq 2u+1.\end{array}\right.$$
Hence,
\begin{align*}\chi_{A(u)_1}(\lambda)&=\sum_{j=0}^{2u+1}b_j\lambda^{2u+1-j}\\
&=\sum_{j=0}^{u}a_j\lambda^{2u+1-j}+(a_{u+1}-(-1)^ia_0)\lambda^{u}\\
&\qquad+\sum_{j=u+2}^{2u+1}(a_j-(-1)^{i}a_{j-u-1}-(-1)^{i-1}a_{j-u-2})\lambda^{2u+1-j}\\
&=\sum_{j=0}^{2u+1}a_j\lambda^{2u+1-j}+(-1)^{i+1}\sum_{j=u+1}^{2u+1}a_{j-u-1}\lambda^{2u+1-j}\\
&+(-1)^{i}\sum_{j=u+2}^{2u+1}a_{j-u-2}\lambda^{2u+1-j}\\
&=\lambda\cdot\chi_{A(u)}(\lambda)+(-1)^{i+1}\sum_{j=0}^{u}a_j\lambda^{u-j}+(-1)^{i}\sum_{j=0}^{u-1}a_j\lambda^{u-1-j}
\end{align*}
By formula \ref{remark3}, we have that
\begin{align*}&(-1)^{i+1}\sum_{j=0}^{u}a_j\lambda^{u-j}+(-1)^{i}\sum_{j=0}^{u-1}a_j\lambda^{u-1-j}\\
=&(-1)^{i+1}\frac{1}{\lambda^{u}}\sum_{j=0}^{u}a_{j}\lambda^{2u-j}
+(-1)^{i}\frac{1}{\lambda^{u+1}}\sum_{j=0}^{u-1}a_{j}\lambda^{2u-j}\\
=&-\frac{1}{(-\lambda)^{3i+2}}
\big(-(\lambda+1)\cdot\sum_{j=0}^{i}(-\lambda)^{3i+3+3j}\big)-\frac{1}{(-\lambda)^{3i+3}}\big(-(\lambda+1)\cdot\sum_{j=0}^{i}(-\lambda)^{3i+3+3j}\big)\\
=&(\lambda+1)\cdot\sum_{j=0}^{i}(-\lambda)^{1+3j}
+(\lambda+1)\cdot\sum_{j=0}^{i}(-\lambda)^{3j}\\
=&(\lambda+1)(1-\lambda)\cdot\sum_{j=0}^{i}(-\lambda)^{3j}\\
=&(\lambda+1)(1-\lambda)\frac{1-(-\lambda)^{3i+3}}{\lambda^3+1}
\end{align*}
Therefore,
\begin{align*}\chi_{A(u)_1}(\lambda)&=\lambda\cdot\chi_{A(u)}(\lambda)+(\lambda+1)(1-\lambda)\frac{1-(-\lambda)^{3i+3}}{\lambda^3+1}\\
&=\frac{\lambda+1}{\lambda^3+1}\big(\lambda(\lambda^{2u+2}-2(-\lambda)^{u+1}+1)
+(1-\lambda)(1-(-\lambda)^{u+1})\big)\\
&=\frac{\lambda+1}{\lambda^3+1}\big(\lambda^{2u+3}+(-\lambda)^{u+2}-(-\lambda)^{u+1}+1\big)\\
&=\frac{(\lambda+1)(\lambda^{u+1}-(-1)^{u+1})(\lambda^{u+2}-(-1)^{u+1})}{\lambda^3+1}
\end{align*}
\end{proof}

\subsection{Coxeter polynomial of $A(u)_v$}

Now we plan to calculate the Coxeter polynomial for the one-branch extension algebra $A(u)_v$.

For a rational function over $\bf{k}$ we refer to the fractional function $\frac{a(\lambda)}{b(\lambda)}$, where $a(\lambda), b(\lambda)\in\bf{k[\lambda]}$. Let $V$ be the set of rational function series $(\phi_j(\lambda))_{j\geq 1}$ satisfying
\begin{equation}\label{three term relation}
\phi_{j+2}(\lambda)=(\lambda+1)\phi_{j+1}(\lambda)-\lambda\phi_{j}(\lambda).
\end{equation}
Then
$V$ forms a $\bf{k}$-vector space, and
any function series $(\phi_j(\lambda))_{j\geq 1}\in V$ is determined by its first two terms $\phi_1(\lambda)$ and $\phi_2(\lambda)$.
It is easy to see that the following function series $(\phi_j(\lambda))_{j\geq 1}$ belong to $V$:
\begin{itemize}
  \item[(1)] $\phi_j(\lambda)=f(\lambda)$ for any $j\geq 1$, where $f(\lambda)$ is a given rational function;
  \item[(2)] $\phi_j(\lambda)=\lambda^j$ for any $j\geq 1$.
\end{itemize}

Moreover, we have the following result.

\begin{lem}\label{jj+1j+2} For any given rational function $f(\lambda)$ and $g(\lambda)$, let $$\phi_j(\lambda)=\lambda^{j} f(\lambda)+g(\lambda), \qquad j\geq 1.$$Then \eqref{three term relation} holds, i.e., $(\phi_j(\lambda))_{j\geq 1}\in V$.
\end{lem}

\begin{proof} By definition we have,
\begin{align*}
&(\lambda+1)\phi_{j+1}(\lambda)-\lambda\phi_{j}(\lambda)\\
=&(\lambda+1)\big(\lambda^{j+1} f(\lambda)+g(\lambda)\big)-\lambda\big(\lambda^{j} f(\lambda)+g(\lambda)\big)\\
=&\lambda^{j+2}f(\lambda)+g(\lambda)\\
=&\phi_{j+2}(\lambda).
\end{align*}
\end{proof}

Now we can state the main result of this section.
\begin{prop}\label{coxeterpoly} Assume $v\geq 1$.
The Coxeter polynomial of $A(u)_v$ is given by
$$\chi_{A(u)_v}(\lambda)=\left\{\begin{array}{ll}
\frac{(\lambda+1)\big(\lambda^{2u+v+2}+(-1)^{u}\sum_{j=2}^{v}\lambda^{u+j}+1\big)}
    {\lambda^3+1}, &u=3i,\\\\
\frac{(\lambda+1)
    \big(\lambda^{2u+v+2}+(-1)^{u+1}\sum_{j=0}^{v}\lambda^{u+j+1}+1\big)}{\lambda^3+1}, &u=3i+1,\\\\
\frac{(\lambda+1)\big(\lambda^{u+1}-(-1)^{u+1}\big)
    \big(\lambda^{u+v+1}-(-1)^{u+1}\big)}{\lambda^3+1},&u=3i+2.\end{array}\right.$$
\end{prop}

\begin{proof}
Let $Q(u)_{v}$ be the bound quiver as follow:
$$
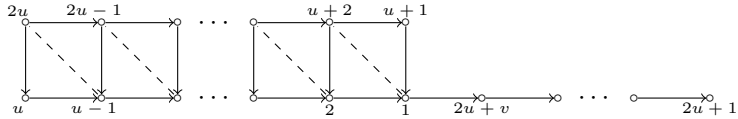
\begin{tikzpicture}
\node()at(-0.1,-0.15){\tiny{$u$}};
\node()at(-0.1,1.15){\tiny{$2u$}};
\node()at(0.9,-0.15){\tiny{$u-1$}};
\node()at(0.9,1.15){\tiny{$2u-1$}};
\node()at(0,0){\tiny{$\circ$}};
\node()at(1,0){\tiny{$\circ$}};
\node()at(1,1){\tiny{$\circ$}};
\node()at(0,1){\tiny{$\circ$}};
\node()at(2,0){\tiny{$\circ$}};
\node()at(2,1){\tiny{$\circ$}};
\node()at(3,0){\tiny{$\circ$}};
\node()at(3,1){\tiny{$\circ$}};
\node()at(4,-0.15){\tiny{$2$}};
\node()at(4,1.15){\tiny{$u+2$}};
\node()at(5,-0.15){\tiny{$1$}};
\node()at(5,1.15){\tiny{$u+1$}};
\node()at(4,0){\tiny{$\circ$}};
\node()at(4,1){\tiny{$\circ$}};
\node()at(5,0){\tiny{$\circ$}};
\node()at(5,1){\tiny{$\circ$}};
\node()at(2.5,0){$\cdots$};
\node()at(2.5,1){$\cdots$};
\draw[-,dashed](0.95,0.05)--(0.05,0.95);
\draw[-,dashed](1.95,0.05)--(1.05,0.95);
\draw[-,dashed](3.95,0.05)--(3.05,0.95);
\draw[-,dashed](4.95,0.05)--(4.05,0.95);
\draw[->](2,0.95)--(2,0.05);
\draw[->](3,0.95)--(3,0.05);
\draw[->](0,0.95)--(0,0.05);
\draw[->](1,0.95)--(1,0.05);
\draw[->](0.05,0)--(0.95,0);
\draw[->](0.05,1)--(0.95,1);
\draw[->](1.05,0)--(1.95,0);
\draw[->](1.05,1)--(1.95,1);
\draw[->](3.05,0)--(3.95,0);
\draw[->](3.05,1)--(3.95,1);
\draw[->](4,0.95)--(4,0.05);
\draw[->](4.05,0)--(4.95,0);
\draw[->](4.05,1)--(4.95,1);
\draw[->](5,0.95)--(5,0.05);
\node()at(6,0){\tiny{$\circ$}};
\node()at(7.5,0){$\cdots$};
\node()at(9,0){\tiny{$\circ$}};
\node()at(7,0){\tiny{$\circ$}};
\node()at(8,0){\tiny{$\circ$}};
\node()at(6,-0.15){\tiny{$2u+v$}};
\node()at(9,-0.15){\tiny{$2u+1$}};
\draw[<-](5.95,0)--(5.05,0);
\draw[<-](6.95,0)--(6.05,0);
\draw[<-](8.95,0)--(8.05,0);
\end{tikzpicture}$$
Let $M$ be the injective module corresponding to the vertex $2u+1$. Then the path algebra $A(u)_{v+1}$ associated to $Q(u)_{v+1}$ is isomorphic to the one-point extension algebra $A(u)_{v}[M]$. On the other hand, the right perpendicular category $M^{\perp}$ in $\textup{mod}\mbox{-}A(u)_{v}$ is derived equivalent to the module category of the algebra associated to $Q(u)_{v-1}$. Thus by Proposition \ref{one-point-extension}, we have \begin{align}\label{chi-ex}\chi_{A(u)_{v+1}}(\lambda)=(1+\lambda)\cdot\chi_{A(u)_{v}}(\lambda)
-\lambda\cdot\chi_{A(u)_{v-1}}(\lambda).\end{align} Hence,
$(\chi_{A(u)_{v+a}}(\lambda))_{v\geq 1}\in V$ for any $a\geq -1$.
We consider the following three cases.

(1) Assume $u=3i$.
By Lemma \ref{jj+1j+2} we obtain
\begin{align*}\phi_{v}(\lambda):&=\frac{(\lambda+1)\big(\lambda^{2u+v+2}+(-1)^{u}\sum_{j=2}^{v}\lambda^{u+j}+1\big)}
    {\lambda^3+1}\\
&=\frac{\lambda+1}{\lambda^3+1}
\big(\lambda^{2u+v+2}+(-1)^{u}\frac{\lambda^{u+2}-\lambda^{u+v+1}}{1-\lambda}+1  \big)\\
&=\frac{\lambda+1}{\lambda^3+1}
\big(\lambda^{2u+2}+\frac{(-\lambda)^{u+1}}{1-\lambda}\big)\cdot\lambda^v
+\frac{\lambda+1}{\lambda^3+1}\big(1+\frac{(-\lambda)^{u+2}}{1-\lambda}\big)\in V.
\end{align*}
Hence it suffices to show that $\chi_{A(u)_{v}}(\lambda)=\phi_{v}(\lambda)$ for $v=1,2$.

By Proposition \ref{d=1} we know that $\chi_{A(u)_{1}}(\lambda)=\phi_{1}(\lambda)$. For $v=2$,
by Proposition \ref{d=0} and Proposition \ref{d=1} and formula \ref{chi-ex}, we get
\begin{align*}\chi_{A(u)_2}&=(\lambda+1)\cdot\chi_{A(u)_1}-\lambda\cdot\chi_{A(u)_0}\\
&=(\lambda+1)\frac{(\lambda+1)(\lambda^{2u+3}+1)}{\lambda^3+1}-\lambda\frac{\lambda+1}{\lambda^3+1}(\lambda^{2u+2}+(-1)^{u+1}\lambda^{u+1}+1)\\
&=\frac{\lambda+1}{\lambda^3+1}\big((\lambda+1)(\lambda^{2u+3}+1)-\lambda(\lambda^{2u+2}+(-1)^{u+1}\lambda^{u+1}+1)\big)\\
&=\frac{\lambda+1}{\lambda^3+1}(\lambda^{2u+4}+(-1)^{u}\lambda^{u+2}+1)\\
&=\phi_{2}(\lambda).
\end{align*}
We are done.

(2) If $u=3i+1$, by similar argument as in (1) we know that
$$\phi_{v}(\lambda):=\frac{(\lambda+1)
    \big(\lambda^{2u+v+2}+(-1)^{u+1}\sum_{j=0}^{v}\lambda^{u+j+1}+1\big)}{\lambda^3+1}\in V,$$
and $\chi_{A(u)_{1}}(\lambda)=\phi_{1}(\lambda)$ by Proposition \ref{d=1}. Moreover,
\begin{align*}\chi_{A(u)_2}&=(\lambda+1)\cdot\chi_{A(u)_1}-\lambda\cdot\chi_{A(u)_0}\\
&=(\lambda+1)\frac{(\lambda+1)(\lambda^{u+1}+(-1)^{u+1})(\lambda^{u+2}+(-1)^{u+1})}{\lambda^3+1}\\
&\qquad-\lambda\frac{\lambda+1}{\lambda^3+1}(\lambda^{2u+2}+(-1)^{u+1}\lambda^{u+1}+1)\\
&=\frac{\lambda+1}{\lambda^3+1}[(\lambda+1)(\lambda^{2u+3}+(-1)^{u+1}\lambda^{u+2}+(-1)^{u+1}\lambda^{u+1}+1)\\
&\qquad-\lambda(\lambda^{2u+2}+(-1)^{u+1}\lambda^{u+1}+1)]\\
&=\frac{\lambda+1}{\lambda^3+1}(\lambda^{2u+4}+(-1)^{u+1}\lambda^{u+3}
+(-1)^{u+1}\lambda^{u+2}+(-1)^{u+1}\lambda^{u+1}+1)\\
&=\phi_{2}(\lambda).
\end{align*}
Then we are done.

(3) If $u=3i+2$, it is easy to see that
$$\phi_{v}(\lambda):=\frac{(\lambda+1)\big(\lambda^{u+1}-(-1)^{u+1}\big)
    \big(\lambda^{u+v}-(-1)^{u+1}\big)}{\lambda^3+1}\in V.$$
    By Propositions \ref{d=0} and \ref{d=1} we have $\phi_{1}(\lambda)=\chi_{A(u)}(\lambda)$ and $\phi_{2}(\lambda)=\chi_{A(u)_{1}}(\lambda)$. Observe that both of $(\phi_{v}(\lambda))_{v\geq 1}$ and $(\chi_{A(u)_{v-1}}(\lambda))_{v\geq 1}$ belong to $V$. Hence $\chi_{A(u)_{v}}(\lambda)=\phi_{v+1}(\lambda)$ for any $v\geq1$.
\end{proof}

\section{Derived equivalences between one-branch extension algebras and Nakayama algebras}

In this section, we will show that four different kinds of one-branch extension algebras $A(u)_v,\ A(u)^v,\ _vA(u)$ and $^vA(u)$ are all derived equivalent to the same Nakayama algebra $N(2u+v,u+v+1)$, then finish the proof of the main Theorem \ref{theorem}. As an application, we obtain the Coxeter polynomials for a half of Nakayama algebras, namely, those Nakayama algebras $N(n,r)$ with $2r\geq n+2$. As a byproduct, we find a new kind of symmetry between Nakayama algebras $N(2r-1,r)$ and $N(2r-1,r+1)$.

\subsection{Derived equivalences between one-branch extension algebras of $A(u)$}

Let $N(n,r)$ $(2\leq r\leq n-1)$ be the Nakayama algebra given as the path algebra of the equioriented quiver
$$\xymatrix{
  1& \ar[l]_-{x}  2&  \ar[l]_-{x}  3&  \ar[l]_-{x}  \cdots  &\ar[l]_-{x}  n-1& \ar[l]_-{x}  n  }$$
of type $\vec{A}_n$ subject to all relations $x^r=0$.

\begin{prop}\label{main1} For $u\geq 1, v\geq 1$,
we have the following derived equivalences:
$$\textup{D}^{b}(^vA(u))\simeq\textup{D}^{b}(A(u)_v)\simeq\textup{D}^{b}(N(2u+v,u+v+1)).$$
\end{prop}

\begin{proof}
Consider the poset $X$ corresponding to $^vQ(u)$ as follow:
$$
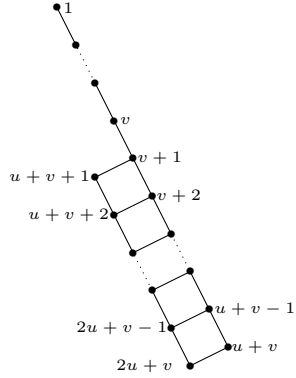
\begin{tikzpicture}
\node()at(7.5,-1.75){\tiny{$\bullet$}};
\node()at(6.9,-1.75){\tiny{$2u+v$}};
\node()at(7.25,-1.25){\tiny{$\bullet$}};
\node()at(6.6,-1.25){\tiny{$2u+v-1$}};
\node()at(6.5,0.25){\tiny{$\bullet$}};
\node()at(5.9,0.25){\tiny{$u+v+2$}};
\node()at(6.25,0.75){\tiny{$\bullet$}};
\node()at(5.65,0.75){\tiny{$u+v+1$}};
\draw[-](7.5,-1.75)--(7.25,-1.25);
\draw[-](7.25,-1.25)--(7,-0.75);
\draw[-](6.5,0.25)--(6.75,-0.25);
\draw[-](6.5,0.25)--(6.25,0.75);
\node()at(6.75,-0.25){\tiny{$\bullet$}};
\node()at(7,-0.75){\tiny{$\bullet$}};
\draw[dotted](7,-0.75)--(6.75,-0.25);
\draw[-](7.5,-1.75)--(8,-1.5);
\draw[-](7.25,-1.25)--(7.75,-1);
\draw[-](6.5,0.25)--(7,0.5);
\draw[-](6.75,1)--(6.25,0.75);
\draw[-](7,-0.75)--(7.5,-0.5);
\draw[-](6.75,-0.25)--(7.25,0);
\node()at(8,-1.5){\tiny{$\bullet$}};
\node()at(8.35,-1.5){\tiny{$u+v$}};
\node()at(7.75,-1){\tiny{$\bullet$}};
\node()at(8.35,-1){\tiny{$u+v-1$}};
\node()at(7,0.5){\tiny{$\bullet$}};
\node()at(7.35,0.5){\tiny{$v+2$}};
\node()at(6.75,1){\tiny{$\bullet$}};
\node()at(7.1,1){\tiny{$v+1$}};
\draw[-](8,-1.5)--(7.75,-1);
\draw[-](7.75,-1)--(7.5,-0.5);
\draw[-](7,0.5)--(7.25,0);
\draw[-](7,0.5)--(6.75,1);
\draw[dotted](7.25,0)--(7.5,-0.5);
\node()at(7.25,0){\tiny{$\bullet$}};
\node()at(7.5,-0.5){\tiny{$\bullet$}};
\node()at(6.5,1.5){\tiny{$\bullet$}};
\node()at(6.65,1.5){\tiny{$v$}};
\node()at(5.75,3){\tiny{$\bullet$}};
\node()at(5.9,3){\tiny{$1$}};
\draw[-](6.5,1.5)--(6.25,2);
\draw[-](6.5,1.5)--(6.75,1);
\draw[-](5.75,3)--(6,2.5);
\draw[dotted](6.25,2)--(6,2.5);
\node()at(6.25,2){\tiny{$\bullet$}};
\node()at(6,2.5){\tiny{$\bullet$}};
\end{tikzpicture}$$
We have $$\textup{D}^b(^vA(u))\simeq\textup{D}^b(I(X)).$$

Let $Y=\{1,2,\cdots,v\}$, which is a closed subset of $X$.
By Proposition \ref{Ladkani}, we have
$$\textup{D}^b(^vA(u))\simeq\textup{D}^b(I(X))\simeq\textup{D}^b(A_Y)\simeq\textup{D}^b(A(u)_v).$$
Similarly, let $Z=\{1,2,\cdots,v+u\}$, which is also a closed subset of $X$.
Note that $A_Z$ is the Nakayama algebra $N(2u+v,u+v+1)$. By Proposition \ref{Ladkani}, we have
$$\begin{array}{ll}\textup{D}^b(^vA(u))\simeq\textup{D}^b(I(X))
\simeq\textup{D}^b(A_Z)\simeq\textup{D}^{b}(N(2u+v,u+v+1)).\end{array}$$
\end{proof}

\begin{prop}\label{main2}
For $u\geq 1, v\geq 1$, we have the following derived equivalences:
$$\textup{D}^{b}(_vA(u))\simeq\textup{D}^{b}(A(u)^v)\simeq\textup{D}^{b}(N(2u+v,u+v+1))$$
\end{prop}

\begin{proof}
We consider the weighted projective line $\mathbb{X}$ of weight type $(2,3,u+v+1)$. Let $$T_1=\big(\bigoplus_{i=0}^{u+v-1} E\langle i\vec{x}_3\rangle\big)\oplus \big(\bigoplus_{i=0}^{u+v-1}E\langle\vec{x}_2+i\vec{x}_3\rangle\big),$$ which has the following shape:
$$\tiny{\xymatrix{E\ar[r]\ar[d]&E\langle\vec{x}_3\rangle\ar[r]\ar[d]&
\cdots\ar[r]&E\langle(u+v-2)\vec{x}_3\rangle\ar[r]\ar[d]&E\langle(u+v-1)\vec{x}_3\rangle\ar[d]\\
E\langle{\vec{x}_2}\rangle\ar[r]&E\langle\vec{x}_2+\vec{x}_3\rangle\ar[r]&\cdots\ar[r]&
E\langle\vec{x}_2+(u+v-2)\vec{x}_3\rangle\ar[r]&E\langle\vec{x}_2+(u+v-1)\vec{x}_3\rangle.}}$$
According to \cite[Theorem 6.1]{KLM2013},
$T_1$ is a tilting object in $\textup{\underline{vect}}\mbox{-}\mathbb{X}$, whose endomorphism algebra is isomorphic to $A(u+v)$.

Let $$T_2=\big(\bigoplus_{i=0}^{u+v-1} E\langle i\vec{x}_3\rangle\big)\oplus \big(\bigoplus_{i=0}^{u+v-1}\mathbb{S}(E\langle{i\vec{x}_3}\rangle)\big),$$
which has the following shape:
$$\tiny{\xymatrix{E\ar[r]&E\langle\vec{x}_3\rangle\ar[r]&\cdots\ar[r]
&E\langle(u+v-2)\vec{x}_3\rangle\ar[r]&E\langle(u+v-1)\vec{x}_3\rangle\ar[dllll]\\
\mathbb{S}E\ar[r]&\mathbb{S}(E\langle\vec{x}_3\rangle)\ar[r]
&\cdots\ar[r]&\mathbb{S}(E\langle(u+v-2)\vec{x}_3\rangle)\ar[r]
&\mathbb{S}(E\langle(u+v-1)\vec{x}_3\rangle).}}$$
According to \cite[Proposition 6.9]{KLM2013},
$T_2$ is also a tilting object in $\textup{\underline{vect}}\mbox{-}\mathbb{X}$, whose endomorphism algebra is isomorphic to the Nakayama algebra $N(2u+2v,u+v+1)$.

Therefore, we have a sequence of equivalences
$$\textup{D}^{b}(A(u+v))\simeq
\textup{\underline{vect}}\mbox{-}\mathbb{X}\simeq\textup{D}^{b}(N(2u+2v,u+v+1)).$$

Let $\mathcal{C}$ be the triangulated subcategory of $\textup{\underline{vect}}\mbox{-}\mathbb{X}$ generated by the exceptional sequence $(E,E\langle\vec{x}_3\rangle,\cdots,E\langle(v-1)\vec{x}_3\rangle)$.
By considering the left perpendicular category of $\mathcal{C}$ in $\textup{\underline{vect}}\mbox{-}\mathbb{X}$, we obtain
\begin{equation}\label{_vA and Nakayama}\textup{D}^{b}(_vA(u))
\simeq^{\perp}\mathcal{C}\simeq\textup{D}^{b}(N(2u+v,u+v+1)).
\end{equation}

Let $\mathcal{D}$ be the triangulated subcategory of $\textup{\underline{vect}}\mbox{-}\mathbb{X}$ generated by the exceptional sequence
$(E\langle\vec{x}_2+u\vec{x}_3\rangle, E\langle\vec{x}_2+(u+1)\vec{x}_3\rangle, \cdots, E\langle\vec{x}_2+(u+v-1)\vec{x}_3\rangle)$.
Then we have
\begin{equation}\label{D perp and A_v}
\mathcal{D}^{\perp}\simeq\textup{D}^{b}(A(u)^v).
\end{equation}
On the other hand, by \cite[Proposition 4.15]{KLM2013}, for any $u\leq l\leq u+v-1$ we have
$$E\langle\vec{x}_2+l\vec{x}_3\rangle=E\langle l\vec{x}_3\rangle(\vec{x}_1-\vec{x}_2).$$
Since the grading shift functor $(\vec{x}_1-\vec{x}_2)$ and the Serre functor $\mathbb{S}$ are both auto-equivalences on $\textup{\underline{vect}}\mbox{-}\mathbb{X}$, we obtain the following equivalences:
\begin{equation}\label{D perp and Nakayama}\begin{array}{ll}\mathcal{D}^{\perp}&\simeq\{E\langle l\vec{x}_3\rangle|u\leq l\leq u+v-1\}^{\perp}\\
&\simeq\{\mathbb{S}(E\langle l\vec{x}_3\rangle)|u\leq l\leq u+v-1\}^{\perp}\\
&\simeq\textup{D}^{b}(N(2u+v,u+v+1)).\end{array}
\end{equation}

Combining with \eqref{_vA and Nakayama}, \eqref{D perp and A_v} and \eqref{D perp and Nakayama}
we have $$\textup{D}^{b}(_vA(u))\simeq\textup{D}^{b}(N(2u+v,u+v+1))\simeq\textup{D}^{b}(A(u)^v).$$
The proof is complete.
\end{proof}

Combining with Proposition \ref{main1} and Proposition \ref{main2}, we finish the proof of Theorem \ref{theorem}.

Observe that the Coxeter polynomial is an invariant under derived equivalences, that is, two derived equivalent algebras share the same Coxeter polynomials. As an immediate consequence of Theorem \ref{theorem}, we have the following result.

\begin{cor} For any $u\geq 1, v\geq 1$, we have
$$\chi_{^vA(u)}(\lambda)=\chi_{A(u)_v}(\lambda)=\chi_{_vA(u)}(\lambda)=\chi_{A(u)^v}(\lambda).$$
\end{cor}

\subsection{Coxeter polynomials of Nakayama algebras}

In Propositions \ref{main1} and \ref{main2}, we have established derived equivalences between one-branch extension algebras and Nakayama algebras. As an application, we can obtain the Coxeter polynomials for all the Nakayama algebras $N(n,r)$ with $2r\geq n+2$.

\begin{prop}\label{corollaryofpolynomial}
Assume that $2r\geq n+2$. The Coxeter polynomial $\chi_{N(n,r)}(\lambda)$ of the Nakayama algebra $N(n,r)$ is given as below.
\begin{itemize}
\item[(1)] For $2r=n+2$,
$$\chi_{N(n,r)}(\lambda)=\left\{\begin{array}{lcl}\frac{(\lambda+1)\big(\lambda^{3n-3r+6}-(-1)^{n-r}\big)}
{(\lambda^3+1)\big(\lambda^{n-r+2}-(-1)^{n-r}\big)},&&n-r\not\equiv 1\,(\mod 3),\\\\
\frac{(\lambda+1)\big(\lambda^{n-r+2}-(-1)^{n-r}\big)^2}{\lambda^3+1},&&n-r\equiv 1\,(\mod 3).\end{array}\right.$$
\item[(2)] For $2r\geq n+3$,
$$\chi_{N(n,r)}(\lambda)=\left\{\begin{array}{lcl}\frac{(\lambda+1)
\big(\lambda^{n+2}+(-1)^{n-r}\sum\limits_{j=0}^{2r-n-2}\lambda^{n-r+2+j}+1\big)}{\lambda^3+1},
&&n-r\equiv 0\,(\mod 3),\\\\
\frac{(\lambda+1)\big(\lambda^{n-r+2}-(-1)^{n-r}\big)\big(\lambda^{r}-(-1)^{n-r}\big)}{\lambda^3+1},
&&n-r\equiv 1\,(\mod 3),\\\\
\frac{(\lambda+1)\big(\lambda^{n+2}+(-1)^{n-r+1}\sum_{j=2}^{2r-n-2}\lambda^{n-r+1+j}+1\big)}{\lambda^3+1},
&&n-r\equiv 2\,(\mod 3).\end{array}\right.$$
\end{itemize}
\end{prop}

\begin{proof}
By Proposition \ref{main1}, for any $u\geq1, v\geq 1$, we have $$\textup{D}^{b}(A(u)_v))\simeq\textup{D}^{b}(N(2u+v,u+v+1)).$$ Let $r=u+v+1$ and $n=2u+v$. That is, $u=n-r+1$ and $v=2r-n-2$. Hence,
$$\textup{D}^{b}(A(n-r+1)_{2r-n-2})\simeq\textup{D}^{b}(N(n,r)).$$ Then the results follow from  Proposition \ref{d=0} for $2r=n+2$ and from Proposition \ref{coxeterpoly} for $2r\geq n+3$, respectively.
\end{proof}

As a byproduct, we also find a new symmetry between Nakayama algebras.

\begin{prop}
Assume that $r\geq 1$. We have the following derived equivalence:
$$\textup{D}^{b}(N(2r-1,r))\simeq\textup{D}^{b}(N(2r-1,r+1)).$$
\end{prop}

\begin{proof}
Consider the poset $X$ with the following Hasse diagram:
$$
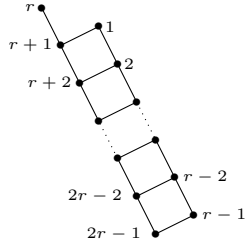
\begin{tikzpicture}
\node()at(5.5,-1.75){\tiny{$\bullet$}};
\node()at(4.95,-1.75){\tiny{$2r-1$}};
\node()at(5.25,-1.25){\tiny{$\bullet$}};
\node()at(4.7,-1.25){\tiny{$2r-2$}};
\node()at(4.5,0.25){\tiny{$\bullet$}};
\node()at(4.1,0.25){\tiny{$r+2$}};
\node()at(4.25,0.75){\tiny{$\bullet$}};
\node()at(3.85,0.75){\tiny{$r+1$}};
\node()at(4,1.25){\tiny{$\bullet$}};
\node()at(3.85,1.25){\tiny{$r$}};
\draw[-](4,1.25)--(4.25,0.75);
\draw[-](5.5,-1.75)--(5.25,-1.25);
\draw[-](5.25,-1.25)--(5,-0.75);
\draw[-](4.5,0.25)--(4.75,-0.25);
\draw[-](4.5,0.25)--(4.25,0.75);
\draw[dotted](5,-0.75)--(4.75,-0.25);
\node()at(5,-0.75){\tiny{$\bullet$}};
\node()at(4.75,-0.25){\tiny{$\bullet$}};
\draw[-](5.5,-1.75)--(6,-1.5);
\draw[-](5.25,-1.25)--(5.75,-1);
\draw[-](4.5,0.25)--(5,0.5);
\draw[-](4.75,1)--(4.25,0.75);
\draw[-](5,-0.75)--(5.5,-0.5);
\draw[-](4.75,-0.25)--(5.25,0);
\node()at(6,-1.5){\tiny{$\bullet$}};
\node()at(6.4,-1.5){\tiny{$r-1$}};
\node()at(5.75,-1){\tiny{$\bullet$}};
\node()at(6.15,-1){\tiny{$r-2$}};
\node()at(5,0.5){\tiny{$\bullet$}};
\node()at(5.15,0.5){\tiny{$2$}};
\node()at(4.75,1){\tiny{$\bullet$}};
\node()at(4.9,1){\tiny{$1$}};
\draw[-](6,-1.5)--(5.75,-1);
\draw[-](5.75,-1)--(5.5,-0.5);
\draw[-](5,0.5)--(5.25,0);
\draw[-](5,0.5)--(4.75,1);
\draw[dotted](5.25,0)--(5.5,-0.5);
\node()at(5.25,0){\tiny{$\bullet$}};
\node()at(5.5,-0.5){\tiny{$\bullet$}};\end{tikzpicture}$$
Let $I(X)$ be the incidence algebra of $X$. Combining with Proposition \ref{main2}, we obtain
\begin{equation}\label{_1A(r-1) and I(X)}
\textup{D}^b(I(X))\simeq\textup{D}^b(_1A(r-1))\simeq\textup{D}^b(N(2r-1,r+1)).
\end{equation}
Let $Y=\{1,2,\cdots,r-1\}$ be the closed subset of $X$.
Then $A_Y$ is isomorphic to the Nakayama algebra $N(2r-1,r)$. By Proposition \ref{Ladkani}, we have
\begin{equation}\label{N(2r-1,r) and I(X)}
\textup{D}^b(I(X))\simeq\textup{D}^b(A_Y)\simeq\textup{D}^b(N(2r-1,r)).
\end{equation}
Then the result follows from \eqref{_1A(r-1) and I(X)} and \eqref{N(2r-1,r) and I(X)}.
\end{proof}

\section{Tilting complexes for Nakayama algebras}

In this section, we give realizations of the one-branch extension algebras $A(u)^v$, $A(u)_v$, $^vA(u)$ and $_vA(u)$ by tilting complexes in the bounded derived category $\textup{D}^b(N(2u+v,u+v+1))$ respectively.

\subsection{Homomorphisms between projective complexes for $N(n,r)$}

Denote by $\mathcal{A}=\textup{mod}\mbox{-}N(n,r)$ for the Nakayama algebra $N(n,r)$. Denote by $\mathcal{C}, \mathcal{K}, \mathcal{D}$ the complex category, the homotopy category and the bounded derived category of $\mathcal{A}$, respectively.

The representation theory of the algebras $N(n,r)$ is well-understood. They are all representation finite and even simply connected.
The Auslander-Reiten quiver of $N(n,r)$ has the following shape:
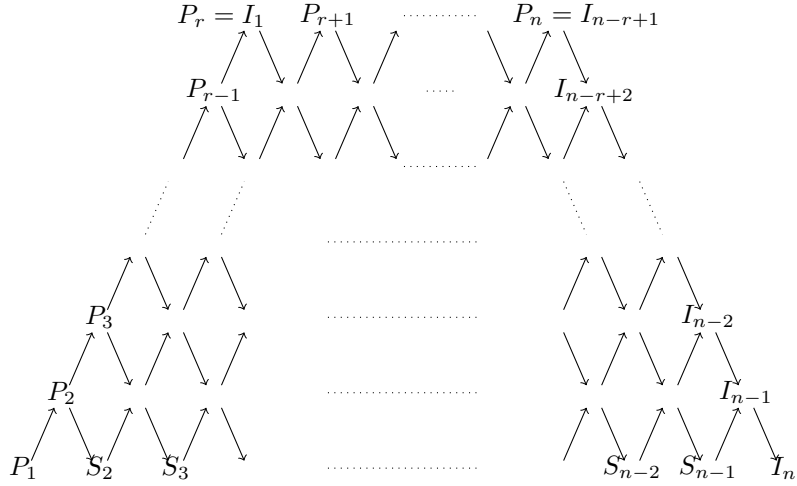
\begin{figure}[H]
\begin{tikzpicture}
\node()at(0,0){$P_1$};
\draw[->](0.1,0.1)--(0.4,0.8);
\node()at(0.5,1){$P_2$};
\draw[->](0.6,1.1)--(0.9,1.8);
\draw[->](0.6,0.8)--(0.9,0.1);
\draw[->](1.1,1.8)--(1.4,1.1);
\draw[->](1.6,0.8)--(1.9,0.1);
\draw[->](2.1,1.8)--(2.4,1.1);
\draw[->](1.6,2.8)--(1.9,2.1);
\draw[->](2.6,2.8)--(2.9,2.1);
\draw[->](2.6,0.8)--(2.9,0.1);
\draw[->](2.6,4.8)--(2.9,4.1);
\draw[->](3.1,5.8)--(3.4,5.1);
\draw[->](1.1,2.1)--(1.4,2.8);
\draw[dotted](1.6,3.1)--(1.9,3.8);
\draw[->](2.1,4.1)--(2.4,4.8);
\draw[->](2.6,5.1)--(2.9,5.8);
\node()at(2.5,5){$P_{r-1}$};
\node()at(2.6,6){$P_r=I_1$};
\node()at(1,0){$S_2$};
\node()at(2,0){$S_3$};
\node()at(4,6){$P_{r+1}$};
\node()at(7.4,6){$P_{n}=I_{n-r+1}$};
\node()at(7.5,5){$I_{n-r+2}$};
\node()at(9,2){$I_{n-2}$};
\node()at(9.5,1){$I_{n-1}$};
\node()at(10,0){$I_{n}$};
\node()at(9,0){$S_{n-1}$};
\node()at(8,0){$S_{n-2}$};
\draw[->](7.6,4.8)--(7.9,4.1);
\draw[dotted](8.1,3.8)--(8.4,3.1);
\draw[dotted](7.1,3.8)--(7.4,3.1);
\draw[->](8.6,2.8)--(8.9,2.1);
\draw[->](9.1,1.8)--(9.4,1.1);
\draw[->](9.6,0.8)--(9.9,0.1);
\draw[->](8.1,2.1)--(8.4,2.8);
\draw[->](7.1,2.1)--(7.4,2.8);
\draw[->](8.6,1.1)--(8.9,1.8);
\draw[->](9.1,0.1)--(9.4,0.8);
\draw[->](7.6,1.1)--(7.9,1.8);
\draw[->](8.1,0.1)--(8.4,0.8);
\draw[->](7.1,0.1)--(7.4,0.8);
\draw[->](7.6,2.8)--(7.9,2.1);
\draw[->](8.1,1.8)--(8.4,1.1);
\draw[->](8.6,0.8)--(8.9,0.1);
\draw[->](7.1,1.8)--(7.4,1.1);
\draw[->](7.6,0.8)--(7.9,0.1);
\draw[->](4.1,5.8)--(4.4,5.1);
\draw[->](6.1,5.8)--(6.4,5.1);
\draw[->](6.6,5.1)--(6.9,5.8);
\draw[->](6.6,4.8)--(6.9,4.1);
\draw[->](7.1,4.1)--(7.4,4.8);
\draw[->](4.6,4.8)--(4.9,4.1);
\draw[->](4.1,4.1)--(4.4,4.8);
\draw[->](3.6,4.8)--(3.9,4.1);
\draw[->](6.1,4.1)--(6.4,4.8);
\draw[->](7.1,5.8)--(7.4,5.1);
\draw[dotted](5,6)--(6,6);
\draw[dotted](4,0)--(6,0);
\draw[dotted](4,1)--(6,1);
\draw[dotted](4,2)--(6,2);
\draw[dotted](4,3)--(6,3);
\draw[dotted](5,4)--(6,4);
\draw[->](4.6,5.1)--(4.9,5.8);
\draw[dotted](5.3,5)--(5.7,5);
\node()at(1,2){$P_3$};
\draw[->](1.1,0.1)--(1.4,0.8);
\draw[->](2.1,0.1)--(2.4,0.8);
\draw[->](1.6,1.1)--(1.9,1.8);
\draw[->](2.6,1.1)--(2.9,1.8);
\draw[->](2.1,2.1)--(2.4,2.8);
\draw[dotted](2.6,3.1)--(2.9,3.8);
\draw[->](3.1,4.1)--(3.4,4.8);
\draw[->](3.6,5.1)--(3.9,5.8);
\end{tikzpicture}
\caption{Auslander-Reiten quiver of Nakayama algbra}\label{ARofNakayama}
\end{figure}
\noindent
where $P_a, I_a, S_a\; (1\leq a\leq n)$ denote the indecomposable projective, injective and simple $N(n,r)$ modules respectively. Denote by $S_a^{(j)}$ the $N(n,r)$ module of length $j$ with top $S_a$.

Observe that $$\textup{dim}_{\mathbf{k}}\,\textup{Hom}_{\mathcal{A}}(P_a,P_b)=\left\{\begin{array}{lcl}
1&&a\leq b< a+r,\\
0&&\textup{otherwise}.\end{array}\right.$$
We choose a $\mathbf{k}$-basis $\xi_{a}$ of the one-dimensional space $\textup{Hom}_{\mathcal{A}}(P_a,P_{a+1})$ for $1\leq a\leq n-1$. Denote by $\xi_{a,b}=\xi_{b-1}\cdots\xi_{a+1}\xi_{a}$ and $\xi_{a,a}=\textup{id}_{P_a}$. Then $\xi_{a,b}$ is a $\mathbf{k}$-basis of $\textup{Hom}_{\mathcal{A}}(P_a,P_{b})$ if $a\leq b< a+r$, and it is zero otherwise. Moreover, $\xi_{c,b}\xi_{a,c}=\xi_{a,b}$ for $a\leq c\leq b$.

We consider indecomposable complex $X^{\bullet}=(X_i,d^i)_{i\in\mathbb{Z}}$ in $\mathcal{D}$ with each $X_i$ an indecomposable projective $N(n,r)$ module or zero. Then $X^{\bullet}$ has the following form (up to isomorphism)
\begin{equation}\label{complexform}\xymatrix{\cdots \ar[r]&P_{a_{i-1}}\ar[rr]^{\xi_{a_{i-1},a_i}}&& P_{a_i}\ar[rr]^{\xi_{a_{i},a_{i+1}}}&& P_{a_{i+1}}\ar[r]&\cdots}\tag{$\ast$}\end{equation}
It is easy to see that $a_{i-1}<a_i<a_{i-1}+r\leq a_{i+1}$ since $X^{\bullet}$ is an indecomposable complex in $\mathcal{D}$. For simplicity, we denote by $X^{\geq i}$ the indecomposable complex of the form (\ref{complexform}) with $X_j=0$ for all $j<i$, and denote by $Y^{\leq i}$ the indecomposable complex of the form (\ref{complexform}) with $Y_j=0$ for all $j>i$.

\begin{lem}\label{lem1}
Let $X^{\geq 0}$ and $Y^{\leq 0}$ be two indecomposable complexes of the form (\ref{complexform}). If $\textup{Hom}_{\mathcal{A}}(X_0,Y_{-1})\neq 0$ or $\textup{Hom}_{\mathcal{A}}(X_1,Y_{0})\neq 0$, then $\textup{Hom}_{\mathcal{K}}(X^{\geq0},Y^{\leq0})=0$.
\end{lem}

\begin{proof}
For any non-zero map $f=(f_i)_{i\in\mathbb{Z}}\in\textup{Hom}_{\mathcal{C}}(X^{\geq0}, Y^{\leq0})$, it is clear that $f_i=0$ for each $i\neq 0$, and $f_0=c\xi_{a_0,b_0}$ for some $c\neq0$.
Consider the following diagram:
$$\xymatrix{X^{\geq0}\ar[d]_{f}:\;\;&\cdots\ar[r]&0\ar[r]\ar[d]&P_{a_{0}}\ar[r]^{\xi_{a_0,a_1}}\ar[d]_{f_0}
\ar@{-->}[dl]_{\alpha}&P_{a_{1}}\ar@{-->}[dl]_{\beta}\ar[r]\ar[d]&\cdots\\
Y^{\leq0}\;\;&\cdots\ar[r]&P_{b_{-1}}\ar[r]_{\xi_{b_{-1},b_0}}&P_{b_{0}}\ar[r]&0\ar[r]&\cdots.}$$
If $\textup{Hom}_{\mathcal{A}}(X_0,Y_{-1})\neq 0$, it follows $a_0\leq b_{-1}<a_0+r$. We take $\alpha=c\xi_{a_0,b_{-1}}$ and $\beta=0$. Since $\xi_{a_0,b_0}=\xi_{b_{-1},b_0}\xi_{a_0,b_{-1}}$, we obtain that $f$ is homotopic to zero. Similarly, if $\textup{Hom}_{\mathcal{A}}(X_1,Y_{0})\neq 0$, then by taking $\alpha=0$ and $\beta=c\xi_{a_1,b_0}$, we also obtain that $f$ is null-homotopic. Hence $\textup{Hom}_{\mathcal{K}}(X^{\geq0},Y^{\leq0})=0$.
\end{proof}

\begin{lem}\label{lem2}
Let $X^{\geq -1}$ and $Y^{\leq 0}$ be two indecomposable complexes of the form (\ref{complexform}). Assume that $\textup{Hom}_{\mathcal{A}}(X_0,Y_{-1})\neq 0$. If one of the following conditions holds:
$$\begin{array}{llllll}
(i)\,\textup{Hom}_{\mathcal{A}}(X_{-1},Y_{0})\neq 0; &(ii)\,\textup{Hom}_{\mathcal{A}}(X_{-1},Y_{-2})\neq 0;&(iii)\;\textup{Hom}_{\mathcal{A}}(X_{1},Y_{0})\neq 0,
\end{array}$$
then $\textup{Hom}_{\mathcal{K}}(X^{\geq -1},Y^{\leq0})=0$.
\end{lem}

\begin{proof}
Observe that if $X_{-1}=0$ or $Y_0=0$, then the result holds by Lemma \ref{lem1}. Hence we assume $X_{-1}\neq0$ and $Y_0\neq0$. For any non-zero map $f=(f_i)_{i\in\mathbb{Z}}\in\textup{Hom}_{\mathcal{C}}(X^{\geq -1},Y^{\leq0})$, it is clear that $f_i=0$ for $i\neq 0,-1$, $f_{-1}=c_{-1}\xi_{a_{-1}, b_{-1}}$ and $f_0=c_0\xi_{a_0,b_0}$ for some scalars $c_{-1}, c_0$.
Consider the following diagram:
$$\xymatrix{X^{\geq -1}\ar[d]_{f}:\;\;&\cdots\ar[r]&0\ar[r]\ar[d]&P_{a_{-1}}\ar[r]^{\xi_{a_{-1},a_0}}\ar[d]_{f_{-1}}\ar@{-->}[dl]_{\alpha}&P_{a_{0}}\ar[r]^{\xi_{a_0,a_1}}\ar[d]_{f_0}\ar@{-->}[dl]_{\beta}&P_{a_{1}}\ar@{-->}[dl]_{\gamma}\ar[r]\ar[d]&\cdots\\
Y^{\leq0}:\;\;&\cdots\ar[r]&P_{b_{-2}}\ar[r]_{\xi_{b_{-2},b_{-1}}}&P_{b_{-1}}\ar[r]_{\xi_{b_{-1},b_0}}&P_{b_{0}}\ar[r]&0\ar[r]&\cdots}$$
For case (i), we have $a_{-1}\leq b_0< a_{-1}+r$ and then $\xi_{a_{-1},b_0}\neq 0$. Since $$c_{-1}\xi_{a_{-1},b_0}=\xi_{b_{-1},b_0}f_{-1}=f_0\xi_{a_{-1},a_0}=c_{0}\xi_{a_{-1},b_0},$$
we have $c_0=c_{-1}$. Observe that $\textup{Hom}_{\mathcal{A}}(X_0,Y_{-1})\neq 0$, then we have $a_{0}\leq b_{-1}$. Hence $f$ is null-homotopic by taking $\beta=c_0\xi_{a_0,b_{-1}}$ and $\alpha=\gamma=0$. For case (ii), by taking $\alpha=(c_{-1}-c_0)\xi_{a_{-1},b_{-2}}$, $\beta=c_0\xi_{a_0,b_{-1}}$ and $\gamma=0$, we obtain that $f$ is null-homotopic. For case (iii), $f$ is null-homotopic by taking $\alpha=0$, $\beta=c_{-1}\xi_{a_0,b_{-1}}$ and $\gamma=(c_0-c_{-1})\xi_{a_1,b_{0}}$.

In each case, $\textup{Hom}_{\mathcal{K}}(X^{\geq -1},Y^{\leq0})=0$. We are done.
\end{proof}

In the following we plan to construct tilting complexes (concentrating at positions $0,\pm1$) in the bounded derived category $\mathcal{D}$. For simplicity, we denote by
$(\xymatrix{X_{{-1}}\ar[r]^{\alpha_{-1}}&X_{{0}}\ar[r]^{\alpha_{0}}&X_{{1}}})$
the following projective complex $$\xymatrix{\cdots\ar[r]&0\ar[r]&X_{{-1}}\ar[r]^{\alpha_{-1}}&X_{{0}}
\ar[r]^{\alpha_{0}}&X_{{1}}\ar[r]&0\ar[r]&\cdots}.$$

\begin{lem}\label{endt}
Let $X^{\bullet}=(\xymatrix{X_{{-1}}\ar[r]^{\alpha_{-1}}&P_{a_{0}}\ar[r]^{\alpha_{0}}&X_{{1}}})$ and $Y^{\bullet}=(\xymatrix{Y_{{-1}}\ar[r]^{\beta_{-1}}&P_{a_{0}}\ar[r]^{\beta_{0}}&Y_{{1}}})$ be two indecomposable complexes in $\mathcal{D}$ of the form (\ref{complexform}).
Then $\textup{Hom}_{\mathcal{K}}(X^{\bullet},Y^{\bullet})\neq 0$ if and only if the following two conditions hold:
\begin{itemize}
\item[(1)] $X_{-1}=0$ or $\textup{Hom}_{\mathcal{A}}(X_{-1},Y_{-1})\neq 0$,
\item[(2)] $Y_{1}=0$ or $\textup{Hom}_{\mathcal{A}}(X_{1},Y_{1})\neq 0$.
\end{itemize}
Moreover, in this case, $\textup{dim}_{{\bf k}}\textup{Hom}_{\mathcal{K}}(X^{\bullet},Y^{\bullet})=1$.
\end{lem}

$$\xymatrix{X^{\bullet}\ar[d]_{f}:\;\;&(X_{{-1}}\ar[r]^{\alpha_{-1}}\ar[d]_{f_{-1}}&P_{a_{0}}
\ar[r]^{\alpha_{0}}\ar[d]_{f_{0}}&X_{{1}})\ar[d]_{f_{1}}\\
Y^{\bullet}:\;\;&(Y_{{-1}}\ar[r]_{\beta_{-1}}&P_{a_{0}}\ar[r]_{\beta_{0}}&Y_{{1}})}$$

\begin{proof}
``$\Longrightarrow$'':\;Let $f=(f_{-1},f_0,f_1)\in\textup{Hom}_{\mathcal{K}}(X^{\bullet},Y^{\bullet})$ be a non-zero  homomorphism from $X^{\bullet}$ to $Y^{\bullet}$. We claim that $f_0\neq0$. If not, we have $f_{-1}\neq 0$ or $f_1\neq0$. Without loss of generality, we assume $f_1\neq0$, then we can write $X_1=P_{a_{1}}$, $Y_1=P_{b_{1}}$ and $f_1=c_1\xi_{a_1,b_1}\neq0$. Since $X^{\bullet}, Y^{\bullet}$ are indecomposable in $\mathcal{D}$, we have $\alpha_0=l\xi_{a_0,a_1}\neq0$ and $\beta_0\neq0$, which implies that $a_0< b_1<a_0+r$. Hence $f_1\alpha_0=c_1l\xi_{a_0,b_1}\neq0$, contradicting with $f_1\alpha_0=\beta_0f_0=0$. This proves the claim.

Therefore, we can assume that $f_0=c\textup{id}_{P_{a_0}}$ for some $c\neq0$. If $X_{-1}\neq0$, then $\alpha_{-1}\neq0$ since $X^{\bullet}$ is indecomposable. Hence
$$\beta_{-1}f_{-1}=f_0\alpha_{-1}=c\alpha_{-1}\neq0.$$
It follows that $f_{-1}\neq0$. Hence $\textup{Hom}_{\mathcal{A}}(X_{-1},Y_{-1})\neq 0$. Similarly, if $Y_1\neq0$, then we have $f_1\neq0$. Hence $\textup{Hom}_{\mathcal{A}}(X_{1},Y_{1})\neq 0$.

``$\Longleftarrow$'': If $\textup{Hom}_{\mathcal{A}}(X_{-1},Y_{-1})\neq 0$ and $\textup{Hom}_{\mathcal{A}}(X_{1},Y_{1})\neq 0$, then we can assume that $X_{\pm1}=P_{a_{\pm1}}$ and $Y_{\pm1}=P_{b_{\pm1}}$ with $a_{-1}\leq b_{-1}< a_0< a_1\leq b_1$. Up to isomorphism, we can assume that $\alpha_i=\xi_{a_{i},a_{i+1}}$ and $\beta_i=\xi_{b_i,b_{i+1}}$ for $i=-1,0$, where $b_0=a_0$. Since $\textup{Hom}_{\mathcal{A}}(P_{a_1},P_{a_0})=0=\textup{Hom}_{\mathcal{A}}(P_{a_0},P_{b_{-1}})$, then $(f_{-1},f_0,f_1)=(\xi_{a_{-1},b_{-1}},\textup{id}_{P_{a_0}},\xi_{a_{1},b_{1}})$ is not null-homotopic. Hence $\textup{Hom}_{\mathcal{K}}(X^{\bullet},Y^{\bullet})\neq 0$. If $X_{-1}=0$ or $Y_1=0$, we can replace $f_{-1}$ or $f_1$ by zero in the above triple $(f_{-1},f_0,f_1)$ respectively, yielding a non-zero homomorphism in $\textup{Hom}_{\mathcal{K}}(X^{\bullet},Y^{\bullet})$.

Moreover, if $\textup{Hom}_{\mathcal{K}}(X^{\bullet},Y^{\bullet})\neq0$, then by the above analysis we know that $f_{_{\pm1}}$ is determined by $f_0$, hence
$\textup{dim}_{{\bf k}}\textup{Hom}_{\mathcal{K}}(X^{\bullet},Y^{\bullet})=1$.
\end{proof}

\subsection{Tilting realizations of $A(u)^v$ and $_vA(u)$ in $\textup{D}^{b}(N(2u+v,u+v+1))$}

In this subsection, we will construct tilting objects in $\mathcal{D}$ (consisting of $N(2u+v,u+v+1)$-modules) to realize $A(u)^v$ and $_vA(u)$.

Recall that $P_i, I_i, S_i\; (1\leq i\leq 2u+v)$ denote the indecomposable projective, injective and simple $N(2u+v,u+v+1)$ modules respectively, and $S_i^{(j)}$ denotes the $N(2u+v,u+v+1)$ module of length $j$ with top $S_i$.
\begin{prop}\label{tiltingobject}
We have that $$T(u)^{v}=(\bigoplus_{i={u+v}}^{2u+v}P_i)\oplus(\bigoplus_{j=1}^{u+v-1}S_{u+v}^{(j)})$$ is a tilting object in $\textup{D}^{b}(N(2u+v,u+v+1))$, whose endomorphism algebra is isomorphic to the one-branch extension algebra $A(u)^{v}$.
\end{prop}

\begin{proof}
Observe that the projective resolution of $S_{u+v}^{(j)}$ in $\mathcal{A}$ has the following form:
\begin{equation}\label{projective resolution}\xymatrix{0\ar[r]&P_{u+v-j}\ar[rr]^{\xi_{u+v-j,u+v}}&&
P_{u+v}\ar[r]&S_{u+v}^{(j)}\ar[r]&0.}\end{equation}
We denote by $E^{(j)}$ the complex $(\xymatrix{P_{u+v-j}\ar[rr]^{\xi_{u+v-j,u+v}}&&P_{u+v}\ar[r]&0)}$, where $P_{u+v}$ is on $0$-position. Then $S_{u+v}^{(j)}\cong E^{(j)}$ in $\mathcal{D}$.

Let $u+v\leq i\leq 2u+v$ and $1\leq j,j'\leq u+v-1$. Then we have
\begin{align*}\textup{Hom}_{\mathcal{K}}(P_i,T(u)^{v}[k])\subseteq\textup{Hom}_{\mathcal{C}}(P_i,T(u)^{v}[k])=0; \quad\forall k\neq 0;\\
\textup{Hom}_{\mathcal{K}}(E^{(j)},P_i[k])\subseteq\textup{Hom}_{\mathcal{C}}(E^{(j)},P_i[k])=0;
\quad\forall k\neq 0,1;\\
\textup{Hom}_{\mathcal{K}}(E^{(j)},E^{(j')}[k])
\subseteq\textup{Hom}_{\mathcal{C}}(E^{(j)},E^{(j')}[k])=0;\quad\forall k\neq 0,1;
\end{align*}
Moreover, by Lemma \ref{lem1} we have
\begin{align*}\textup{Hom}_{\mathcal{K}}(E^{(j)},P_i[1])=0=
\textup{Hom}_{\mathcal{K}}(E^{(j)},E^{(j')}[1]).\end{align*}
Therefore, for any $k\neq 0$, by Lemma \ref{fibrant} we have
\begin{align*}\textup{Hom}_{\mathcal{D}}(T(u)^{v},T(u)^{v}[k])
=\textup{Hom}_{\mathcal{K}}(T(u)^{v},T(u)^{v}[k])=0.\end{align*}

Moreover, by \eqref{projective resolution} we know that each indecomposable projective module $P_i\;(1\leq i\leq u+v)$ can be generated by $T(u)^{v}$. Hence $T(u)^{v}$ generates the bounded derived category $\textup{D}^{b}(N(2u+v,u+v+1))$. Therefore, $T(u)^{v}$ is a tilting object in  $\textup{D}^{b}(N(2u+v,u+v+1))$.

It is easy to see that the tilting object has the following shape:
$$\tiny{\xymatrix{P_{u+v}\ar[r]\ar[d]&S_{u+v}^{(u+v-1)}\ar[r]\ar[d]&\cdots\ar[r]&S_{u+v}^{(v+2)}\ar[r]\ar[d]&S_{u+v}^{(v+1)}\ar[d]\ar[r]&\cdots\ar[r]&S_{r-1}^{(1)}\\
P_{u+v+1}\ar[r]&P_{u+v+2}\ar[r]&\cdots\ar[r]&P_{2u+v-1}\ar[r]&P_{2u+v}}}$$
Hence the endomorphism algebra $\End_{\mathcal{A}}(T(u)^v)$ is isomorphic to the one-branch extension algebra $A(u)^{v}$. We are done.
\end{proof}

Dually, 
let $$_vT(u)=(\bigoplus_{i={1}}^{u+1}I_i)\oplus(\bigoplus_{j=1}^{u+v-1}S_{u+j}^{(j)}),$$ which has the following shape: $$\tiny{\xymatrix{&&I_{1}\ar[r]\ar[d]&I_{2}\ar[r]\ar[d]&\cdots\ar[r]&I_{u-1}\ar[r]\ar[d]&I_{u}\ar[d]\\
S_{u+1}^{(1)}\ar[r]&\cdots\ar[r]&S_{u+v+1}^{(v+1)}\ar[r]&S_{u+v+2}^{(v+2)}\ar[r]&\cdots\ar[r]&S_{2u+v-1}^{(u+v-1)}\ar[r]&I_{u+1}}}$$
Then we have the following result.

\begin{prop}\label{tiltingobject2}
$_vT(u)$ is a tilting object in $\textup{D}^{b}(N(2u+v,u+v+1))$, whose endomorphism algebra is isomorphic to the one-branch extension algebra $_{v}A(u)$.
\end{prop}

In the following picture, we point out the positions of the indecomposable direct summands of $T(u)^v$ (\emph{resp.} $_vT(u)$) sitting in the Auslander-Reiten quiver of $\mathcal{A}=\textup{mod}\mbox{-}N(2u+v,u+v+1)$.
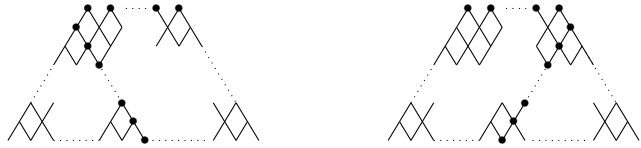
\begin{figure}[H]
\begin{center}
\begin{tikzpicture}
\draw[-](0.15,0.25)--(0.3,0.5);
\draw[-](0.45,0.25)--(0.3,0.5);
\draw[-](0.45,0.75)--(0.3,0.5);
\draw[-](0.45,0.25)--(0.6,0.5);
\draw[-](0.45,0.75)--(0.6,0.5);
\draw[-](0.75,0.25)--(0.6,0.5);
\draw[-](0.75,0.75)--(0.6,0.5);
\draw[dotted](0.45,0.75)--(0.75,1.25);
\draw[dotted](0.75,0.25)--(1.35,0.25);
\draw[-](0.75,1.25)--(0.9,1.5);
\draw[-](1.05,1.75)--(0.9,1.5);
\draw[-](1.2,2)--(1.05,1.75);
\node()at(1.05,1.75){\tiny{$\bullet$}};
\node()at(1.2,2){\tiny{$\bullet$}};
\draw[-](1.2,2)--(1.35,1.75);
\draw[-](1.2,1.5)--(1.05,1.75);
\draw[-](1.2,1.5)--(1.35,1.75);
\node()at(1.2,1.5){\tiny{$\bullet$}};
\node()at(1.35,1.25){\tiny{$\bullet$}};
\node()at(1.5,2){\tiny{$\bullet$}};
\draw[-](0.9,1.5)--(1.05,1.25);
\draw[-](1.2,1.5)--(1.05,1.25);
\draw[-](1.2,1.5)--(1.35,1.25);
\draw[-](1.5,1.5)--(1.35,1.25);
\draw[dotted](1.65,0.75)--(1.35,1.25);
\node()at(1.65,0.75){\tiny{$\bullet$}};
\node()at(1.8,0.5){\tiny{$\bullet$}};
\node()at(1.95,0.25){\tiny{$\bullet$}};
\draw[-](1.65,0.75)--(1.8,0.5);
\draw[-](1.95,0.25)--(1.8,0.5);
\draw[-](1.65,0.75)--(1.5,0.5);
\draw[-](1.65,0.25)--(1.8,0.5);
\draw[-](1.35,0.25)--(1.5,0.5);
\draw[-](1.65,0.25)--(1.5,0.5);
\draw[-](1.5,2)--(1.35,1.75);
\draw[-](1.5,1.5)--(1.35,1.75);
\draw[-](1.5,2)--(1.65,1.75);
\draw[-](1.5,1.5)--(1.65,1.75);
\draw[dotted](1.7,2)--(2,2);
\node()at(2.1,2){\tiny{$\bullet$}};
\node()at(2.4,2){\tiny{$\bullet$}};
\draw[-](2.1,2)--(2.25,1.75);
\draw[-](2.1,1.5)--(2.25,1.75);
\draw[-](2.4,2)--(2.25,1.75);
\draw[-](2.4,1.5)--(2.25,1.75);
\draw[-](2.4,2)--(2.55,1.75);
\draw[-](2.4,1.5)--(2.55,1.75);
\draw[-](2.7,1.5)--(2.55,1.75);
\draw[dotted](2.7,1.5)--(3.15,0.75);
\draw[-](3.45,0.25)--(3.15,0.75);
\draw[-](3.15,0.25)--(3.3,0.5);
\draw[-](3.15,0.25)--(3,0.5);
\draw[-](2.85,0.25)--(3.15,0.75);
\draw[-](3,0.5)--(2.85,0.75);
\draw[dotted](2.05,0.25)--(2.75,0.25);
\draw[-](5.15,0.25)--(5.3,0.5);
\draw[-](5.45,0.25)--(5.3,0.5);
\draw[-](5.45,0.75)--(5.3,0.5);
\draw[-](5.45,0.25)--(5.6,0.5);
\draw[-](5.45,0.75)--(5.6,0.5);
\draw[-](5.75,0.25)--(5.6,0.5);
\draw[-](5.75,0.75)--(5.6,0.5);
\draw[dotted](5.45,0.75)--(5.75,1.25);
\draw[dotted](5.75,0.25)--(6.35,0.25);
\draw[-](5.75,1.25)--(5.9,1.5);
\draw[-](6.05,1.75)--(5.9,1.5);
\draw[-](6.2,2)--(6.05,1.75);
\node()at(7.55,1.75){\tiny{$\bullet$}};
\node()at(6.2,2){\tiny{$\bullet$}};
\draw[-](6.2,2)--(6.35,1.75);
\draw[-](6.2,1.5)--(6.05,1.75);
\draw[-](6.2,1.5)--(6.35,1.75);
\node()at(7.4,1.5){\tiny{$\bullet$}};
\node()at(7.25,1.25){\tiny{$\bullet$}};
\node()at(6.5,2){\tiny{$\bullet$}};
\draw[-](7.25,1.25)--(7.4,1.5);
\draw[-](7.25,1.25)--(7.1,1.5);
\draw[-](7.55,1.25)--(7.7,1.5);
\draw[-](7.55,1.25)--(7.4,1.5);
\draw[-](7.85,1.25)--(7.7,1.5);
\draw[-](5.9,1.5)--(6.05,1.25);
\draw[-](6.2,1.5)--(6.05,1.25);
\draw[-](6.2,1.5)--(6.35,1.25);
\draw[-](6.5,1.5)--(6.35,1.25);
\draw[dotted](6.95,0.75)--(7.25,1.25);
\node()at(6.95,0.75){\tiny{$\bullet$}};
\node()at(6.8,0.5){\tiny{$\bullet$}};
\node()at(6.65,0.25){\tiny{$\bullet$}};
\draw[-](6.95,0.75)--(6.8,0.5);
\draw[-](6.65,0.75)--(6.8,0.5);
\draw[-](6.95,0.25)--(6.8,0.5);
\draw[-](6.65,0.75)--(6.5,0.5);
\draw[-](6.65,0.25)--(6.8,0.5);
\draw[-](6.35,0.25)--(6.5,0.5);
\draw[-](6.65,0.25)--(6.5,0.5);
\draw[-](6.5,2)--(6.35,1.75);
\draw[-](6.5,1.5)--(6.35,1.75);
\draw[-](6.5,2)--(6.65,1.75);
\draw[-](6.5,1.5)--(6.65,1.75);
\draw[dotted](6.7,2)--(7,2);
\node()at(7.1,2){\tiny{$\bullet$}};
\node()at(7.4,2){\tiny{$\bullet$}};
\draw[-](7.1,2)--(7.25,1.75);
\draw[-](7.1,1.5)--(7.25,1.75);
\draw[-](7.4,2)--(7.25,1.75);
\draw[-](7.4,1.5)--(7.25,1.75);
\draw[-](7.4,2)--(7.55,1.75);
\draw[-](7.4,1.5)--(7.55,1.75);
\draw[-](7.7,1.5)--(7.55,1.75);
\draw[dotted](7.7,1.5)--(8.15,0.75);
\draw[-](8.45,0.25)--(8.15,0.75);
\draw[-](8.15,0.25)--(8.3,0.5);
\draw[-](8.15,0.25)--(8,0.5);
\draw[-](7.85,0.25)--(8.15,0.75);
\draw[-](8,0.5)--(7.85,0.75);
\draw[dotted](7.05,0.25)--(7.75,0.25);
\end{tikzpicture}
\end{center}
\caption{$T(u)^{v}$ and $_vT(u)$ in $\textup{mod}\mbox{-}N(2u+v,u+v+1)$}
\end{figure}

\subsection{Tilting realizations of $A(u)_v$ and $^vA(u)$ in $\textup{D}^{b}(N(2u+v,u+v+1))$}

In order to realize $A(u)_v$ and $^vA(u)$ by tilting objects, we need to use projective complexes in $\mathcal{D}$ rather than $N(2u+v,u+v+1)$ modules in $\mathcal{A}$.

Let $$E_i=(\xymatrix{0\ar[r]&P_{u+v+1}\ar[rr]^{\xi_{u+v+1,u+v+1+i}}&&P_{u+v+1+i})},\;\;1\leq i\leq u-1;$$ $$F_j=(\xymatrix{P_{j}\ar[rr]^{\xi_{j,u+v+1}}&&P_{u+v+1}\ar[r]&0)},\;\;u\leq j\leq u+v;$$ and $$G_l=(\xymatrix{P_{l}\ar[rr]^{\xi_{l,u+v+1}}&&P_{u+v+1}\ar[rr]^{\xi_{u+v+1,u+v+1+l}}&&P_{u+v+1+l}}),\;\;1\leq l\leq u-1.$$

\begin{prop}\label{tiltingobject3}
Let $$T(u)_{v}=P_{u+v+1}\oplus(\bigoplus_{i=1}^{u-1} E_i)\oplus(\bigoplus_{j=u}^{u+v} F_j)\oplus(\bigoplus_{l=1}^{u-1} G_l).$$
Then $T(u)_{v}$ is a tilting object in $\textup{D}^{b}(N(2u+v,u+v+1))$, whose endomorphism algebra is isomorphic to the one-branch extension algebra $A(u)_{v}$.
\end{prop}

\begin{proof}
Let $1\leq i,i',l,l'\leq u-1$ and $u\leq j,j'\leq u+v$. Since there are no homomorphisms from $P_a$ to $P_b$ if $a>b$, it is straightforward to check that
\begin{align*}&\textup{Hom}_{\mathcal{K}}(T(u)_{v},T(u)_{v}[k])\subseteq\textup{Hom}_{\mathcal{C}}(T(u)_{v},T(u)_{v}[k])=0; \quad\forall k\neq 0,1,2;\\
&\textup{Hom}_{\mathcal{K}}\big(P_{u+v+1}\oplus E_i,(P_{u+v+1}\oplus F_j)[1]\big)\subseteq\textup{Hom}_{\mathcal{C}}\big(P_{u+v+1}\oplus E_i,(P_{u+v+1}\oplus F_j)[1]\big)=0;\\
&\textup{Hom}_{\mathcal{K}}\big(F_j\oplus G_l,(P_{u+v+1}\oplus F_{j'})[2]\big)\subseteq\textup{Hom}_{\mathcal{C}}\big(F_j\oplus G_l,(P_{u+v+1}\oplus F_{j'})[2]\big)=0;\\
&\textup{Hom}_{\mathcal{K}}\big(P_{u+v+1}\oplus E_i,T(u)_{v}[2]\big)\subseteq\textup{Hom}_{\mathcal{C}}\big(P_{u+v+1}\oplus E_i,T(u)_{v}[2]\big)=0.
\end{align*}
Moreover, by Lemma \ref{lem1} we have
\begin{align*}
&\textup{Hom}_{\mathcal{K}}\big(P_{u+v+1}\oplus E_i,(E_{i'}\oplus G_l)[1]\big)=0;\\
&\textup{Hom}_{\mathcal{K}}\big(F_j\oplus G_l,(P_{u+v+1}\oplus F_{j'})[1]\big)=0;\\
&\textup{Hom}_{\mathcal{K}}\big(F_j\oplus G_l,(E_i\oplus G_{l'})[2]\big)=0.
\end{align*}
And by Lemma \ref{lem2} we have
\begin{align*}
\textup{Hom}_{\mathcal{K}}\big(F_j\oplus G_l,(E_i\oplus G_{l'})[1]\big)=0.
\end{align*}
Therefore, for any $k\neq 0$, by Lemma \ref{fibrant} we have
\begin{align*}\textup{Hom}_{\mathcal{D}}(T(u)_{v},T(u)_{v}[k])
=\textup{Hom}_{\mathcal{K}}(T(u)_{v},T(u)_{v}[k])=0.\end{align*}

Furthermore, by the construction of $T(u)_{v}$ we know that each indecomposable projective module in $\textup{mod}\mbox{-}N(2u+v,u+v+1)$ can be generated by $T(u)_{v}$. Hence $T(u)_{v}$ generates the bounded derived category $\textup{D}^{b}(N(2u+v,u+v+1))$. Therefore, $T(u)_{v}$ is a tilting object in  $\textup{D}^{b}(N(2u+v,u+v+1))$.

By Lemma \ref{endt}, it is easy to see that the tilting object has the following shape:
$$\tiny{\xymatrix{E_1\ar[r]\ar[d]&E_2\ar[r]\ar[d]&\cdots\ar[r]&E_{u-1}\ar[r]\ar[d]&P_{u+v+1}\ar[d]\\
G_1\ar[r]&G_2\ar[r]&\cdots\ar[r]&G_{u-1}\ar[r]&F_u
\ar[r]&\cdots\ar[r]&F_{u+v-1}\ar[r]&F_{u+v}}}$$
Hence the endomorphism algebra $\End_{\mathcal{D}}(T(u)_{v})$ is isomorphic to the one-branch extension algebra $A(u)_{v}$. We are done.
\end{proof}

Dually, we consider the following indecomposable injective complexes $$E_i'=(I_i\rightarrow I_{u}\rightarrow 0),\;\;1\leq i\leq u-1;$$ $$F_j'=(0\rightarrow I_{u}\rightarrow I_j),\;\;u+1\leq j\leq u+v+1;$$ and $$G_l'=(I_l\rightarrow I_{u}\rightarrow I_{u+v+1+l}),\;\;1\leq l\leq u-1.$$ Let
$$^{v}T(u)=I_{u}\oplus(\bigoplus_{i=1}^{u-1} E_i')\oplus(\bigoplus_{j=u+1}^{u+v+1} F_j')\oplus(\bigoplus_{l=1}^{u-1} G_l'),$$ which has the following shape: $$\tiny{\xymatrix{F_{u+1}'\ar[r]&F_{u+2}'\ar[r]&\cdots\ar[r]&F_{u+v+1}'\ar[r]\ar[d]&G_{1}'\ar[r]\ar[d]&\cdots\ar[r]&G_{u-2}'\ar[r]\ar[d]&G_{u-1}'\ar[d]\\
&&&I_u\ar[r]&E_1'\ar[r]&\cdots\ar[r]&E_{u-2}'\ar[r]&E_{u-1}'}}$$
Then we have the following result.

\begin{prop}\label{tiltingobject4}
$^{v}T(u)$ is a tilting object in $\textup{D}^{b}(N(2u+v,u+v+1))$, whose endomorphism algebra is isomorphic to the one-branch extension algebra $^{v}A(u)$.
\end{prop}

\noindent {\bf Acknowledgements.}\quad
This work was supported by the National Natural Science Foundation of China (No. 11871404).

\bibliographystyle{plain}

\end{document}